%% file: main.tex
\documentclass[11pt, a4paper]{amsart}
\title{Monoidal Model structures on filtered chain complexes relating to spectral sequences}
\author[J. A. Brotherston]{James A. Brotherston}
\address{School of Mathematics and Statistics, University of Sheffield, S3 7RH, UK}
\email{brotherston.maths@gmail.com}
\thanks{This work was supported by the Engineering and Physical Sciences Research Council.}
\input{./tex/packages.tex}
\input{./tex/theorems.tex}
\input{./tex/commands.tex}
\begin{document}
\maketitle
\begin{abstract}
  \input{./tex/abstract.tex}
\end{abstract}
\input{./tex/introduction.tex}
\input{./tex/acknowledgements.tex}
\input{./tex/preliminaries.tex}
\input{./tex/construction.tex}
\input{./tex/cofibrations.tex}
\input{./tex/monoidal.tex}
\input{./tex/algebras.tex}
\input{./tex/cofibrant-unit.tex}
\appendix
\input{./tex/monoidal-proof.tex}
\FloatBarrier
\bibliographystyle{alpha}
\bibliography{./tex/bibliography}
\end{document}

%% file: tex/packages.tex
\usepackage{fullpage}
\usepackage{tikz}
\usepackage{tikz-cd}
\usepackage{pict2e}
\usepackage{hyperref}
\hypersetup{pdfauthor={James A. Brotherston},pdftitle={Monoidal Model structures on filtered chain complexes relating to spectral sequences}}
\usepackage{amsmath}
\usepackage{amssymb}
\usepackage{mathdots} 
\usepackage{mathtools}
\usepackage{amsthm}
\usepackage[noabbrev,capitalise]{cleveref}
\usepackage{todonotes}
\usepackage{pdflscape}
\usepackage{afterpage}
\usepackage{placeins}

%% file: tex/theorems.tex
\theoremstyle{plain}
\newtheorem{theo}{Theorem}[subsection]
\Crefname{theo}{Theorem}{Theorems}
\newtheorem{lemm}[theo]{Lemma}
\Crefname{lemm}{Lemma}{Lemmas}
\newtheorem{prop}[theo]{Proposition}
\Crefname{prop}{Proposition}{Propositions}
\newtheorem{coro}[theo]{Corollary}
\Crefname{coro}{Corollary}{Corollaries}
\theoremstyle{definition}
\newtheorem{defi}[theo]{Definition}
\Crefname{defi}{Definition}{Definitions}
\newtheorem{exam}[theo]{Example}
\Crefname{exam}{Example}{Examples}
\theoremstyle{remark}
\newtheorem{rema}[theo]{Remark}
\Crefname{rema}{Remark}{Remarks}
\newtheorem{nota}[theo]{Notation}
\Crefname{nota}{Notation}{Notations}

\theoremstyle{plain}
\newtheorem{theoL}{Theorem}
\Crefname{theoL}{Theorem}{Theorems}

\Crefname{lemmL}{Lemma}{Lemmas}

\newtheorem{propL}[theoL]{Proposition}
\Crefname{propL}{Proposition}{Propositions}

\Crefname{coroL}{Corollary}{Corollaries}

%% file: tex/commands.tex
\newcommand{\mc}{\mathcal}
\newcommand{\Zbb}{\mathbb{Z}}
\newcommand{\wt}{\widetilde}
\DeclarePairedDelimiter\abs{\lvert}{\rvert}
\let\amsamp=&
\renewcommand{\phi}{\varphi}
\newcommand{\ol}{\overline}
\newcommand{\ul}{\underline}
\newcommand{\id}{\mathrm{id}}
\DeclareMathOperator*{\colim}{colim}
\DeclareMathOperator*{\im}{im}
\newcommand{\eHom}{\underline{\mathrm{Hom}}}
\newcommand{\Hom}{{\mathrm{Hom}}}
\newcommand{\dom}{\mathrm{dom}}
\newcommand{\Quill}{\simeq_{Q}}
\newcommand{\I}{\mathbb{S}}
\newcommand{\Lotimes}{\overset{L}{\otimes}}
\newcommand{\Ch}{\mathcal{C}}
\newcommand{\Fib}{\mathrm{Fib}}
\newcommand{\Cofib}{\mathrm{Cofib}}
\newcommand{\W}{\mathcal{W}}
\newcommand{\Ho}{\mathrm{Ho}}
\newcommand{\mathdash}{\text{-}}
\newcommand{\Inj}{\mathdash\mathrm{Inj}}
\newcommand{\Proj}{\mathdash\mathrm{Proj}}
\newcommand{\Cell}{\mathdash\mathrm{Cell}}
\newcommand{\Cof}{\mathdash\mathrm{Cof}}
\newcommand{\fCh}{f\mathcal{C}}
\newcommand{\fdga}{fdga}
\newcommand{\fdgc}{fdgc}
\newcommand{\fChr}{\left(f\mathcal{C}\right)_r}
\newcommand{\fChrp}{\left(f\mathcal{C}\right)_{r'}}
\newcommand{\fChS}{\left(\fCh\right)_S}
\newcommand{\fChHat}{\tilde{\fCh}}
\newcommand{\fChHatS}{\left(\fChHat\right)_S}
\newcommand{\fChCof}{\fCh^\mathrm{Cof}}
\newcommand{\Er}{\mathcal{E}_r}
\newcommand{\rSusp}{\Sigma^r}
\newcommand{\rLoops}{\Omega^r}
\newcommand{\toplus}{\oplus_\tau}
\newcommand{\Shift}{S}
\newcommand{\Dec}{\mathrm{Dec}}
\newcommand{\Supp}[1]{\mathrm{Supp_{#1}}}
\newcommand{\Z}{\mathcal{Z}}
\newcommand{\B}{\mathcal{B}}

\makeatletter
\newcommand{\noloc}{\nobreak\mskip6muplus1mu{:}\nonscript
  \mkern-\thinmuskip\mathpunct{}\mskip2mu\relax}
\newcommand{\smallbot}{%
  \begingroup\setlength\unitlength{.15em}%
  \begin{picture}(1,1)
    \roundcap
    \polyline(0,0)(1,0)
    \polyline(0.5,0)(0.5,1)
  \end{picture}%
  \endgroup
}
\newcommand{\inadj}[4]{
  #1\colon #2%
  \mathrel{\vcenter{%
      \offinterlineskip\m@th
      \ialign{%
        \hfil$##$\hfil\cr
        \longrightarrow\cr
        \noalign{\kern-.3ex}
        \smallbot\cr
        \longleftarrow\cr
      }%
    }}%
  #3 \noloc #4%
}
\newcommand{\inadjarrows}[2]{
  #1%
  \mathrel{\vcenter{%
      \offinterlineskip\m@th
      \ialign{%
        \hfil$##$\hfil\cr
        \longrightarrow\cr
        \noalign{\kern-.3ex}
        \smallbot\cr
        \longleftarrow\cr
      }%
    }}%
  #2%
}
\makeatother
\newcommand{\adj}[4]{
  \begin{tikzcd}[ampersand replacement=\&]
    {#1} \colon {#2} \arrow[r, shift left=.5ex]
    \arrow[r, phantom, "\smallbot"] \& {#3} \noloc {#4}
    \arrow[l, shift left=.5ex]
  \end{tikzcd}%
}

\newlength\stextwidth
\newcommand\makesamewidth[3][c]{%
  \settowidth{\stextwidth}{#2}%
  \makebox[\stextwidth][#1]{#3}%
}

%% file: tex/abstract.tex
We establish monoidal model structures on model categories of filtered chain complexes constructed by Cirici, Egas Santander, Livernet and Whitehouse whose weak equivalences are the quasi-isomorphisms on the $r$-page of the associated spectral sequences.
In doing so we provide a partial classification of cofibrant objects and cofibrations of the model structures involving a boundedness restriction on the filtration.
As a consequence we also obtain, by results of Schwede and Shipley, cofibrantly generated model structures on the categories of filtered differential graded algebras as well as their modules.


%% file: tex/introduction.tex
\section{Introduction}\label{introduction}
Model categories were introduced by Quillen in \cite{Q} as a framework for abstract homotopy theory.
They provide a unified approach in that many homotopy theories can be realised as having the additional structure of a model category, standard examples being: topological spaces, CW-complexes and simplicial sets with $\pi_\ast$-isomorphisms as weak equivalences which carry model category structures yielding equivalent homotopy categories; and chain complexes with quasi-isomorphisms as weak equivalences which have \textit{projective} and \textit{injective} model structures yielding the derived category as its homotopy category.
See \cite{Balchin} for an extensive catalogue of model structures.

Since their introduction model categories have provided a powerful setting in which to carry out homotopical calculations, providing notions of cofibrant and fibrant resolutions of objects which are homotopically well behaved, constructions of path and cylinder objects for constructing homotopy classes of morphisms, and Quillen adjunctions/equivalences for relating model categories to each other. There are further refined notions of model categories providing useful additional structures, e.g.\ cofibrantly generated, left/right proper, stable, monoidal, cellular and simplicial model categories.
See \cite{H} for a standard account.

We will specifically be interested in \textit{monoidal model categories}. These are model categories equipped with a monoidal product which is sufficiently compatible with the homotopy structure so as to descend to a monoidal product on the homotopy category.
This extra structure consists of the \textit{unit} and \textit{pushout-product axioms}.
In \cite{SS} Schwede and Shipley added the \textit{monoid axiom} to these to further obtain model categories of algebra and module objects.
From the projective model structure on chain complexes, which satisfies all three of these axioms, one can deduce a monoidal product on the derived category and a model category of differential graded algebras.

Spectral sequences, the other key object of interest in this paper, are a homological tool used to compute homology groups.
They have become a standard tool in algebraic topology for computations, see e.g.\ the Leray-Serre, Adams, Atiyah-Hirzebruch, Eilenberg-Moore spectral sequences and many more in \cite{M}.

A general setup starts with a chain complex, whose homology one wishes to compute, which is first equipped with a filtration (chosen wisely). This also induces a filtration on the homology giving \textit{graded pieces}. By then approximating the kernel and the image of the differentials using the filtration one forms successive approximations to the graded pieces each being the homology of the previous.
From a computational point of view one generally hopes there is enough information to compute any differentials appearing in these approximations and that the spectral sequence \textit{converges} to the homology -- we do not discuss this notion in any detail here and direct the reader to \cite{Boardman} for information on the subject.

\subsection{Homotopy theories and spectral sequences in the literature}
Model category structures on categories yielding associated spectral sequences are desirable in situations where objects are constructed up to a quasi-isomorphism on some page of an associated spectral sequence.
The introduction of \cite{CELW} lists a few such examples in the contexts of \textit{the mixed Hodge theory of complex algebraic varieties} and \textit{the rational homotopy theory} of Sullivan.
To these we add \cite{Sagave} which constructs minimal derived $A_\infty$-models of a differential graded algebra well defined up to an equivalence on the $2$-page of a spectral sequence.
With regard to a morphism of spectral sequence we will say it is an \textit{$r$-quasi-isomorphism} if it is an isomorphism between the $(r+1)$-pages.

The first work along these lines is that of Halperin and Tanr\'e in \cite{HT} motivated by rational homotopy theory.
Working with a category of filtered differential graded commutative algebras they define notions of \textit{$(R,r)$-extensions} and \textit{$(R,r)$-fibrations} and use them to show any morphism of such algebras has a \textit{model} in the sense there's a factorisation into an $(R,r)$-extension followed by a $r$-quasi-isomorphism.
This is done by constructing a model at the level of the $r$-page of the associated spectral sequences and \textit{perturbing} to obtain a model on the level of algebras.

In their thesis, \cite{Cirici}, Cirici studies the category of filtered differential graded algebras over a field in the context of a \textit{$P$-category with cofibrant models} and with weak equivalences the $r$-quasi-isomorphisms.
We highlight in our paper similarities between our results and theirs regarding cofibrant objects.

Muro and Roitzheim construct two model categories on bicomplexes in \cite{MR} whose weak equivalences are in the first case detected by the totalisation functor and in the second are $1$-quasi-isomorphisms.

Cirici, Egas Santander, Livernet and Whitehouse construct model categories for each $r\geq 0$ on filtered chain complexes and bicomplexes whose weak equivalences are the $r$-quasi-isomorphisms in \cite{CELW}.
Their bicomplex model structures were later generalised to similar model structures on multicomplexes in \cite{FGLW}.
This current paper further studies these model structures on filtered chain complexes.

More recently in \cite{LW} Livernet and Whitehouse began a study of the homotopy theory of spectral sequences without a chosen source of spectral sequences, i.e.\ working with a category of spectral sequences instead of those of filtered chain complexes or multicomplexes for example, and obtain a \textit{Brown category} structure whose weak equivalences are similarly the $r$-quasi-isomorphisms. 

\subsection{The work of Cirici, Egas Santander, Livernet and Whitehouse}
For the following we fix a commutative unital ground ring $R$ to work over.
In \cite{CELW} the authors construct model category structures on the categories of filtered chain complexes and bicomplexes whose weak equivalences are quasi-isomorphisms on the $r$-page of the associated spectral sequences.
We recall their setup here for the case of filtered chain complexes with the major results also stated in \cref{CELWresults}.

In the category $\fCh$ of filtered chain complexes the authors define objects $\Z_r(p,n)$ and $\B_r(p,n)$ for each $r\geq 0$ and $p,n\in\Zbb$ along with morphisms $\phi_r\colon \Z_r(p,n)\rightarrow \B_r(p,n)$.
For an $A\in\fCh$ these are representing objects for the \textit{$r$-cycles} $Z_r^{p,p+n}(A)$ and \textit{$r$-boundaries} $B_r^{p,p+n}(A)$ in that the associated $r$-page in bidegree $(p,p+n)$ is given by:
\begin{equation*}
  E_r^{p,p+n}(A)\coloneqq
  \frac{Z_r^{p,p+n}(A)}{B_r^{p,p+n}(A)}\cong
  \frac{\Hom_{\fCh}\left(\Z_r(p,n),A\right)}{\Hom_{\fCh}\left(\B_r(p,n),A\right)}\;.
\end{equation*}

This is in analogy to the case for (unfiltered) chain complexes where we have sphere and disc objects, and a map from the former to the latter, representing the kernel and boundary of the differentials with homology realised as the quotient of the analogous $R$-modules of homomorphisms.

Using the morphisms $I_r\coloneqq\left\{\phi_{r+1}\colon\Z_{r+1}(p,n)\rightarrow \B_{r+1}(p,n)\right\}_{p,n\in\Zbb}$ as generating cofibrations and the morphisms $J_r\coloneqq\left\{0\rightarrow\Z_r(p,n)\right\}_{p,n\in\Zbb}$ as generating acyclic cofibrations they establish a right proper cofibrantly generated model structure $\fChr$ for each $r\geq 0$ whose weak equivalences are those morphisms of filtered chain complexes which induce an isomorphism between the $(r+1)$-pages of the associated spectral sequences.

The authors also introduce right proper cofibrantly generated model structures $\fChrp$ with generating cofibrations $I_{r'}\coloneqq I_r\cup\bigcup_{k=0}^{r}J_k$ and generating acyclic cofibrations $J_{r'}\coloneqq \bigcup_{k=0}^rJ_r$.
All these model structures are shown to be Quillen equivalent via the shift-d\'ecalage adjunction of Deligne.

In this paper we will build on the model structures of \cite{CELW}.

\subsection{Contents of this paper}

We first show there are further model structures $\fChS$ lying in between those of the $\fChr$ and $\fChrp$, where $S\subseteq\left\{0,1,\ldots,r\right\}$ containing $r$, as follows.
Let $I_S\coloneqq I_r\cup\bigcup_{k\in S}J_k$ and $J_S\coloneqq\bigcup_{k\in S}J_k$.
In \cref{filtChainsS} we show the following result.
\begin{theoL}
  There is a right proper cofibrantly generated model structure on $\fCh$ denoted $\fChS$ with generating cofibrations $I_S$ and generating acyclic cofibrations $J_S$.
\end{theoL}
An interpretation of how, for a fixed $r$, the model categories determined by $S$ vary is given by a property of their cofibrant objects: \cref{cofibrantkPageDifferentialIsZero} tells us that for $A\in\fChS$ cofibrant, the $k$-page differential of $A$ is necessarily $0$ when $0\leq k<r $ and $k\notin S$.

Our main result is the following theorem, which is show in \cref{fChSIsMonoidal}.
\begin{theoL}\label{introTheoremMonoidal}
  The model categories $\fChS$ are closed symmetric monoidal model categories.
\end{theoL}
As an immediate corollary we establish in \cref{homotopyCategoryfChSIsMonoidal} that their homotopy categories are closed symmetric monoidal categories.
By verifying Schwede and Shipley's monoid axiom we also establish model categories of algebra and module objects in filtered chain complexes in \cref{leftAModulesModCat,AModulesModCat,AAlgebrasModCat}.
\begin{theoL}\label{introTheoremModCatsAlgMod}
  There are cofibrantly generated model categories of:
  \begin{itemize}
  \item left $A$-modules where $A$ is a filtered differential graded algebra,
  \item $A$-modules where $A$ is a filtered differential graded commutative algebra, and
  \item $A$-algebras where $A$ is a filtered differential graded commutative algebra.
  \end{itemize}
\end{theoL}

The remainder of this introduction describes results we obtain to establish these results.

In the model structures $\fChr$ of \cite{CELW} we establish some necessary properties of 
the cofibrant objects.
These are similar to the conditions occurring for cofibrant objects in the projective model structure on unbounded chain complexes.
\cref{cofibrantConditions} establishes that for a cofibrant object $A\in\fChr$:
\begin{enumerate}
\item $A^n$ and $A^n/F_pA^n$ are projective for all $p,n\in\Zbb$,
\item the filtration on $A$ is exhaustive,
\item for $a\in F_pA^n$ we have $da\in F_{p-r}A^{n+1}$, and
\item for any morphism $A\rightarrow \rSusp K$ of $A$ into (the $r$-suspension of) an $r$-acyclic $K$ there is a lift against $C_r(K)\rightarrow \rSusp K$.
\end{enumerate}
where $\rSusp$ is an $r$-homotopical suspension functor.

Conversely if one assumes an $A\in\fChr$ satisfies these four conditions along with:
\begin{enumerate}
  \setcounter{enumi}{4}
\item for any cohomological degree $n$ there exists a $p(n)$ such that $F_{p(n)}A^n=0$,
\end{enumerate}
then we have the following result shown in \cref{subclassOfCofibrantObjects}.
\begin{propL}
  If $A\in\fChr$ satisfies conditions 1-5 above then it is cofibrant.
\end{propL}

This only exhibits a subclass of the cofibrant objects however.
The category $\fCh$ is monoidal however the unit is not cofibrant, \cref{UnitIsNotCofibrant}.
A cofibrant replacement of the unit $Q_r\I$ is constructed and shown to be cofibrant (in all the model categories $\fChS$) in \cref{lemm:QrIIsCofibrant} however it does not satisfy condition 5.

We also construct a subclass of the cofibrations of $\fChr$.
A cofibration is necessarily a morphism of filtered chain complexes of the form $i\colon A\rightarrow A\toplus C$ where $C$ is cofibrant in $\fChr$ and $i$ is an inclusion onto the $A$ summand with the $\toplus$ denoting a twisted differential from $C$ to $A$, \cref{defi:twistedDirectSum}.
We say $\tau$ is \textit{$r$-suppressive} if $\tau F_pC\subseteq F_{p-r}A$ for all $p$, see \cref{rSupressive}.

Conversely we have the following result shown in \cref{subclassOfCofibrations}.
\begin{propL}
  Let $i\colon A\rightarrow A\toplus C$ be as above with $C$ cofibrant satisfying condition 5 above and $\tau$ suppressing filtration by $r$, \cref{rSupressive}, then $i$ is a cofibration in $\fChr$.  
\end{propL}

By explicitly computing the pushout-product of two generating cofibrations we can realise the pushout-product as a morphism of this form hence also a cofibration, \cref{PushoutProductWithAJIsCofibration,PushoutProductOfIAndIIsCofibration} -- this computation is mostly left deferred to \cref{monoidal-proof}.
Along with a construction of a cofibrant replacement for the unit in \cref{rCofibrantUnit} and a demonstration that it satisfies the unit axiom \cref{VeryStrongUnitAxiomForFCHS} these results prove \cref{introTheoremMonoidal}, that the $\fChS$ are monoidal model categories.

\cref{introTheoremModCatsAlgMod} constructing further model categories of modules and algebras follows immediately from the monoid axiom demonstrated in \cref{MonoidAxiomInfChS}.

Using work of Muro, \cite{Muro}, on cofibrantly generated model categories we construct some morphisms which when added to the set of generating (acyclic) cofibrations make the unit a cofibrant object in a new model category Quillen equivalent to $\fChS$.

This paper is based on work from the authors's Ph.D.\ thesis, \cite{BThesis}.

\subsection{Structure of the paper}
\begin{itemize}
\item \cref{preliminaries} collects the required preliminaries for this paper including those on filtered chain complexes, spectral sequences, Deligne's d\'ecalage functor, (monoidal) model categories, Muro's procedure for making the unit cofibrant in a monoidal model category and a summary of the main results of \cite{CELW}.
\item \cref{construction} generalises the model categories of \cite{CELW} and constructs $\fChS$ for any finite non-empty subset of $\mathbb{N}$.
\item \cref{cofibrations} establishes necessary conditions for an object of $\fChr$ to be cofibrant and provides a partial converse where we assume a bounded filtration assumption.
  We also study cofibrations in $\fChr$ and show that the shift and d\'ecalage functors preserve the property of being cofibrant.
\item \cref{monoidal} constructs a cofibrant replacement for the unit in $\fChS$ and verifies the (very strong) unit axiom and pushout-product axioms required of a monoidal model category.
\item \cref{algebras} establishes the monoid axiom for the monoidal model categories $\fChS$ and obtains model categories of filtered differential graded algebras and their modules.
\item \cref{cofibrant-unit} uses the results of Muro to modify the model categories to have cofibrant unit.
\item \cref{monoidal-proof} contains a detailed construction of the pushout-product of generating cofibrations in $\fChS$.
\end{itemize}

%% file: tex/acknowledgements.tex
\section*{Acknowledgements}

I would like to express my deep gratitude to my Ph.D.\ supervisor Sarah Whitehouse whose support and guidance this research was carried out under.
I'm also very grateful to Daniel Graves and Jordan Williamson for many helpful conversations.


%% file: tex/preliminaries.tex
\section{Preliminaries}\label{preliminaries}
We work over a fixed commutative unital ring $R$.
Our chain complexes are cohomologically graded and unbounded.
We take $d$ to be the differential of any chain complex considered and adorn with a superscript to distinguish when necessary as in $d^A$ and $d^B$ for chain complexes $A$ and $B$.
The category of chain complexes will be denoted $\Ch$.
\subsection{Filtered chain complexes}
\begin{defi}\label{filtChainsObjsMorphs}
  A \textit{filtered chain complex} $A$ is a chain complex equipped with an \textit{increasing filtration}, i.e.\ subcomplexes $F_pA$ with $F_pA\subseteq F_{p+1}A\subseteq A$ for each $p\in\Zbb$. A \textit{morphism of filtered chain complexes} $f\colon A\rightarrow B$ is a morphism of the underlying chain complexes which preserves the filtration, i.e.\ $f(F_pA)\subseteq F_pB$.
\end{defi}
\begin{defi}\label{filtChainsCat}
  The \textit{category of filtered chain complexes} with objects and morphisms given as in \cref{filtChainsObjsMorphs} will be denoted $\fCh$.
\end{defi}
\begin{defi}
  The filtration on an $A\in\fCh$ is said to be \textit{exhaustive} if $A=\cup_pF_pA$, i.e.\ any element appears in a finite filtration degree.
\end{defi}
\begin{rema}
  The category $\fCh$ is not an Abelian category, \cite[Chapter II \S5.17]{GM}.
\end{rema}
For computations of colimits in $\fCh$ we recall the construction of \cite{CELW}. The category denoted $\Zbb_\infty$ has objects the integers and precisely one morphism $n\rightarrow m$ whenever $n\leq m$ with a terminal object $\infty$ adjoined.
Note then there is an adjunction $\inadj{\rho}{\Ch^{\Zbb_\infty}}{\fCh}{i}$ where $\Ch^{\Zbb_\infty}$ is the functor category from $\Zbb_\infty$ to $\Ch$. The functor $i$ is the inclusion functor with $(iA)(p) = F_pA$ and $(iA)(\infty) = A$.
The functor $\rho$ is given on a functor $X\colon\Zbb_\infty\rightarrow\fCh$ by $F_p\rho X \coloneqq \im\left(X(p)\rightarrow X(\infty)\right)$.
\begin{lemm}[{\cite[Remark 2.6]{CELW}}]\label{filteredChainsLimits}
  The category $\fCh$ is complete and cocomplete. Limits are computed homological degreewise and filtration degreewise. The colimit of a diagram $X\colon \mc{J}\rightarrow\fCh$ is computed by $\rho\colim_{\mc{J}}iX$.\qed
\end{lemm}

\begin{nota}
  The filtered chain complex $A$ which we denote by $R_{(p)}^n$ has as its underlying chain complex a copy of the $R$-module $R$ in cohomological degree $n$ and is otherwise $0$, with filtration such that $F_qA^n=0$ for all $q<p$ and $F_qA^n=A^n=R$ for all $q\geq p$.
  For such an $A$ we say that $A$ is of pure weight $p$ and concentrated in degree $n$.
  This notation is in keeping with that used in \cite{CELW}.
  We will also give a name to the generator of $R_{(p)}^n$ either by writing $1_{(p)}^n$ or in diagrams by writing the filtered graded $R$-module as $R_{(p)}^n\left\{a\right\}$ with $a$ a generator.

  We use this notation to build further objects of $\fCh$ later, e.g. $R_{(p)}^n\rightarrow R_{(p-r)}^{n+1}$ will denote a direct sum of two such objects with an identity differential appearing in filtration degrees p and above.
\end{nota}

The category of filtered chain complexes has a tensor product and internal hom object which after forgetting filtration are those of the category of chain complexes.
For an element $a\in A$ of a (filtered) chain complex we denote by $\abs{a}$ its cohomological degree.
\begin{defi}\label{tensor}
  The \textit{tensor product of filtered chain complexes} $A$ and $B$ has underlying chain complex given in degree $n$ by:
  \begin{equation*}
    \left(A\otimes B\right)^n\coloneqq \bigoplus_{m}A^m\otimes B^{n-m}
  \end{equation*}
  with differential given by $d(a\otimes b) = da\otimes b +(-1)^{\abs{a}}a\otimes db$ for elements $a\in A^m$ and $b\in B^{n-m}$.
  The filtration is given by:
  \begin{equation*}
    F_p\left(A\otimes B\right)\coloneqq
    \sum_i\im\left(F_iA\otimes F_{p-i}B\hookrightarrow A\otimes B\right)\; .
  \end{equation*}
\end{defi}
Since the differential preserves the filtrations of $A$ and $B$ so too does the differential preserve the filtration in the tensor product.
\begin{nota}\label{tensorProductBracketingNotation}
  We will denote by $R_{(p)+(q)}^{(n)+(m)}$ the tensor product of the filtered chain complexes $R_{(p)}^n$ and $R_{(q)}^m$ making judicious use of bracketing to keep the data of the underlying tensor components.
  This tensor product is isomorphic to $R_{(p+q)}^{n+m}$.
\end{nota}
The monoidal product in $\fCh$ yields a definition of filtered differential graded algebras, i.e.\ monoid objects in $\fCh$ as well as their module objects.
\begin{defi}\label{fdgas}
  A \textit{filtered differential graded algebra} $A$ is a filtered chain complex $A$ with a morphism $m\colon A\otimes A\rightarrow A$ of filtered chain complexes satisfying the usual associativity axioms and a morphism $u\colon R_{(0)}^0\rightarrow A$ of filtered chain complexes satisfying the usual left and right unit laws with respect to the multiplication $m$.

  A \textit{filtered differential graded commutative algebra} is a filtered differential graded algebra which in addition satisfies the usual identity $m\circ t= m$ for $t$ the evident twist morphism.
\end{defi}
\begin{defi}\label{categoryOfFDGAS}
  The \textit{category of filtered differential graded algebras} denoted $\fdga$ has objects the filtered differential graded algebras and morphisms those morphisms of $\fCh$ compatible with the multiplication and unit maps in the usual way.

  Similarly the \textit{category of differential graded commutative algebras} denoted $\fdgc$ is the full subcategory of $\fdga$ whose objects have a graded commutative multiplication.
\end{defi}
In particular then given a filtered differential graded algebra $A$ and elements $a\in F_pA^n$ and $b\in F_qA^m$ the product $a\cdot b\coloneqq m(a\otimes b)\in F_{p+q}A^{n+m}$ for finite indices $n,m\in\Zbb$.
If the filtration is not exhaustive and $a\in A^n\setminus\cup_pF_pA^n$ then $a\otimes b\in A^{n+m}\setminus\cup_qF_qA^{n+m}$.
We have the Koszul sign rule $a\cdot b = (-1)^{\abs{a}\abs{b}}b\cdot a$ and the Leibniz rule $d(a\cdot b) = d(a)\cdot b +(-1)^{\abs{a}}a\cdot d(b)$.
Given a filtered differential graded algebra $A$ the $A$-modules are defined similarly.
\begin{defi}\label{filteredDGModules}
  Let $A$ be a filtered differential graded algebra $A$. A \textit{left $A$-module} is a filtered chain complex $M$ and morphism of filtered chain complexes $A\otimes M\rightarrow M$ satisfying the usual module morphism laws.
\end{defi}

We define a filtration on the $R$-module of homomorphisms $\Hom(A^m, B^{m+n})$ by $F_p\Hom(A^m,B^{m+n})\coloneqq \left\{f\;|\: f(F_qA^m)\subseteq F_{p+q}B^{m+n}\right\}$, i.e.\ those morphisms of $R$-modules that at most increase filtration by $p$.
\begin{defi}\label{internalHom}
  The \textit{internal hom object of filtered chain complexes} $A$ and $B$, denoted $\eHom\left(A,B\right)$ has underlying chain complex given in degree $n$ by:
  \begin{equation*}
    \eHom\left(A,B\right)^n\coloneqq
      \prod_{m\in\Zbb}\Hom\left(A^m, B^{m+n}\right)
  \end{equation*}
  with differential given on an element $(f_m)_m$ by:
  \begin{equation*}
    d\colon (f_m)_m\mapsto (d^B\circ f_m - (-1)^nf_{m+1}\circ d^A)_m\,.
  \end{equation*}
  The filtration is given by:
  \begin{equation*}
    F_p\eHom\left(A,B\right)^n\coloneqq
    \prod_{m\in\Zbb}F_p\Hom\left(A^m,B^{m+n}\right)\,.
  \end{equation*}
\end{defi}
\begin{lemm}
  For any $A\in\fCh$ there is an adjunction pair $\left(A\otimes-\right) \dashv\eHom(A,-)$ on the category $\fCh$.
\end{lemm}
\begin{proof}
  This is the usual adjunction after forgetting filtration.
  One need only check it's compatible with filtration.
\end{proof}
\begin{lemm}
  The bifunctors \cref{tensor,internalHom} equip the category of filtered chain complexes with the structure of a closed symmetric monoidal category whose tensor unit is given by $R_{(0)}^0$.\qed
\end{lemm}
The following definitions of the $r$-suspension and $r$-cone can be found in \cite[Definition 3.5]{CELW}.
The former there is referred to as the \textit{$r$-translation} and the sign conventions differ slightly from ours.
\begin{defi}\label{rSuspension}
  For an $A\in\fCh$ the \textit{$r$-suspension} of $A$ denoted $\rSusp A$ is the filtered chain complex with underlying chain complex that of $\Sigma A$ and filtration given by $F_p\rSusp A^n\coloneqq F_{p-r}A^{n+1}$.
  The differential is given by $d^{\rSusp A}=-d^A$.
 Similarly we denote by $\rLoops A$ the \textit{$r$-suspension of $A$} with underlying chain complex $\Omega A$, filtration given by $F_p\rLoops A^n\coloneqq F_{p+r}A^{n-1}$ and differential $d^{\rLoops A}=-d^A$.
\end{defi}

These can be written in terms of the tensor product as:
\begin{align*}
  \rSusp A&\coloneqq R_{(r)}^{-1}\otimes A,       
  &\rLoops A&\coloneqq R_{(-r)}^{1}\otimes A.
\end{align*}

\begin{defi}\label{rCone}
  For a morphism $f\colon A\rightarrow B$ of filtered chain complexes the \textit{$r$-cone} of $f$ denoted $C_r(f)$ is the filtered chain complex with underlying filtered graded module given by $\rSusp A\oplus B$, in particular $F_pC_r(f)\coloneqq F_{p-r}A^{n+1}\oplus F_pB^n$, and differential given by $d\colon (a,b) \mapsto (-da,fa+db)$.
  We denote by $C_r(A)$ the $r$-cone of the identity morphism $\id\colon A\rightarrow A$.
\end{defi}

Lastly for the construction of cofibrantly generated model categories via Quillen's small object argument we will need the following lemma.
\begin{lemm}\label{filtChainsAreSmall}
  All objects of $\fCh$ are small relative to $\fCh$.
\end{lemm}
\begin{proof}
  The proof is similar to that of \cite[Lemma 2.3.2]{H} accounting for the addition of the filtration.
\end{proof}
\subsection{Spectral sequences from filtered chain complexes}
\begin{defi}
  A \textit{spectral sequence} is a sequence of $(\Zbb,\Zbb)$-bigraded differential modules $\left\{E_r^{\bullet,\bullet}, d_r\right\}$ for $r\geq 0$ such that the differentials $d_r$ are of bidegree $(-r,1-r)$, $d_r\colon E_r^{p,q}\rightarrow E_r^{p-r,q+r-1}$, and further that:
  \begin{equation*}
    E_{r+1}^{p,q}\cong H^{p,q}(E_{r}^{\bullet,\bullet},d_R)
    =\frac{\ker\left(d_r\colon E_r^{p,q}\rightarrow E_r^{p-r,q+1-r}\right)}
    {\im\left(d_r\colon E_r^{p+r,q+1-r}\rightarrow E_r^{p,q}\right)}
  \end{equation*}
  for each $r\geq 0$. We refer to the $E_r^{\bullet,\bullet}$ as the \textit{$r$-page} of the spectral sequence.
\end{defi}
From a filtered chain complex one can extract an \textit{associated spectral sequence}.
The following defines the $r$-cycles and $r$-boundary objects whose quotient will give the $r$-page of a spectral sequence.
This definition can also be found in \cite{CELW} and a similar one in \cite{M} although note McCleary's $r$-boundary elements differ.
\begin{defi}\label{associatedSS}
  Let $A$ be a filtered chain complex. We define in bidegree $(p,p+n)$:
  \begin{itemize}
  \item for $r\geq 0$ the \textit{$r$-cycles} of $A$ as:
    \begin{equation*}
      Z_r^{p,p+n}(A)\coloneqq F_pA^n\cap d^{-1}F_{p-r}A^{n+1}\,,
    \end{equation*}
  \item for $r=0$ the \textit{$r$-boundaries} of $A$ as:
    \begin{equation*}
      B_0^{p,p+n}(A)\coloneqq Z_0^{p-1,p-1+n}(A) = F_{p-1}A^n\,,
    \end{equation*}
  \item and for $r\geq 1$ \textit{$r$-boundaries} of $A$ in bidegree $(p,p+n)$ as:
    \begin{equation*}
      B_r^{p,p+n}(A)\coloneqq dZ_{r-1}^{p+r-1,p+r-1+n+1}(A)+Z_{r-1}^{p-1,p-1+n}(A)\,.
    \end{equation*}
  \end{itemize}
  The \textit{$r$-page of the associated spectral sequence} denoted $E_r^{p,p+n}(A)$ is then given by:
  \begin{equation*}
    E_r^{p,p+n}(A)\coloneqq \frac{Z_r^{p,p+n}(A)}{B_r^{p,p+n}(A)}\,.
  \end{equation*}
  The differential $d$ restricts to $r$-cycles and there is then an induced differential on the $r$-page sending a class $[a]$ to $[d_ra]$.
\end{defi}
\begin{lemm}[{\cite[Theorem 2.6]{M}}]
  The $\left\{E_r^{\bullet,\bullet},d_r\right\}$ of \cref{associatedSS} have the structure of a spectral sequence.\qed
\end{lemm}
\begin{defi}
  Those morphisms $f\colon A\rightarrow B$ of filtered chain complexes inducing an isomorphism between the $(r+1)$-pages of the associated spectral sequences will be called \textit{$r$-quasi-isomorphisms} or \textit{$r$-weak equivalences} and the class of such morphisms denoted by $\Er$.
\end{defi}
The $r$-cycles can be viewed as functors from $\fCh$ to $R$-modules and are representable as shown in  \cite[\S 3.1]{CELW}.
The $r$-boundaries are not quite representable as there is a choice of how to decompose an $r$-boundary in the defining sum of the $R$-modules $dZ^{\ast,\ast}_{\ast}(A)+Z^{\ast,\ast}_\ast(A)$, however their is still a representing object for these choices.
We recall these representing objects here.
\begin{defi}
  The representing object for the functor $Z_r^{p,p+n}(-)$ will be denoted $\Z_r(p,n)$ and is given by:
  \begin{equation*}
    \Z_r(p,n)\coloneqq \left(R_{(p)}^n\overset{1}{\longrightarrow} R_{(p-r)}^{n+1}\right)\,.
  \end{equation*}
  The representing object for the functor $B_r^{p,p+n}(-)$, with $r\geq 1$, is denoted $\B_r(p,n)$ and given by
  \begin{equation*}
    \B_r(p,n)\coloneqq \Z_{r-1}(p+r-1,n-1)\oplus \Z_{r-1}(p-1,n)
  \end{equation*}
  or more explicitly by
  \begin{equation*}
    \B_r(p,n)\coloneqq\left(
      R_{(p-r+1)}^{n-1}      \overset{\begin{pmatrix}1\\0\end{pmatrix}}{\longrightarrow} R_{(p)}^n\oplus R_{(p-1)}^n      \overset{\begin{pmatrix}0&1\end{pmatrix}}{\longrightarrow} 
      R_{(p-r)}^{n+1}\right)\,.
  \end{equation*}
\end{defi}
The authors of \cite{CELW} also give a representing morphism exhibiting any $r$-boundary as an $r$-cycle.
\begin{defi}
  The morphism $\phi_r\colon \Z_r(p,n)\rightarrow\B_r(p,n)$ of filtered chain complexes for $r\geq 1$ is depicted (in the vertical direction) by:
  \begin{equation*}
    \begin{tikzcd}
      &\left(R_{(p)}^n\right.\arrow[d,"\Delta"]\arrow[r,"1"]
      &\left.R_{(p-r)}^{n+1}\right)\arrow[d,"1"]\\
      \left(R_{(p+r-1)}^{n-1}\right.
        \arrow[r,"\begin{pmatrix}1\\0\end{pmatrix}"]
      &R_{(p)}^n\oplus R_{(p-1)}^n
      \arrow[r,"{\begin{pmatrix}0\amsamp 1\end{pmatrix}}"]
      & \left.R_{(p-r)}^{n+1}\right)
    \end{tikzcd}\,.
  \end{equation*}
\end{defi}
Recall the $r$-cone construction of \cref{rCone}.
\begin{lemm}[{\cite[Remark 3.6]{CELW}}]\label{rConeIsrAcyclic}
  Let $f\colon A\rightarrow B$ be a morphism of filtered chain complexes.
  Then $f\in\Er$ if and only if $C_r(f)$ is $r$-acyclic.\qed
\end{lemm}
In particular note that for $A\in\fCh$ that $C_r(A)\coloneqq C_r(\id_A)$ is always $r$-acyclic.
\subsection{The shift-d\'ecalage adjunction of Deligne}
The d\'ecalage functor $\Dec$ of the following definition was introduced by Deligne in \cite{Deligne}.
\begin{defi}\label{shiftDecalage}
  Let $r\geq 0$ and $A\in\fCh$. The \textit{shift} and \textit{d\'ecalage} endofunctors $\Shift^r$ and $\Dec^r$ on $\fCh$ are the identity on the underlying chain complex and modify the filtration by
  \begin{align*}
    F_p(\Shift^rA)^n&\coloneqq F_{p+rn}A^n\,,\\
    F_p(\Dec^rA)^n&\coloneqq Z_r^{p-rn,p-rn+n}(A)\,,
  \end{align*}
  and are such that $\Shift^r$ is equal to r composites of $\Shift\coloneqq\Shift^1$ and $\Dec^r$ is r composites of $\Dec\coloneqq\Dec^1$.
\end{defi}
\begin{lemm}[{\cite[\S2.3]{CiriciGuillen}}]\label{shiftDecalageAdjunction}
  For each $r\geq 0$ there is an adjunction $\Shift^r\dashv\Dec^r$ for which the unit $\id \Rightarrow \Dec^r\circ\Shift^r$ is the identity natural transformation.\qed
\end{lemm}
The main property of the shift and d\'ecalage functors is (up to a shift of indexing) there are isomorphisms $E_{r+1}(\Shift A)\cong E_r(A)$ and $E_{r}(\Dec A)\cong E_{r+1}(A)$ for all $r\geq 0$, see \cite[Proposition 1.3.4]{Deligne}.
In the context of the model structures $\fChr$ and $\fChrp$ of \cite{CELW} the authors used this property and the fact that $\Dec$ sends $(r+1)$-fibrations to $r$-fibrations to demonstrate Quillen equivalences $\fChr\Quill\left(\fCh\right)_{r+l}$ and $\fChrp\Quill\left(\fCh\right)_{(r+l)'}$ for all $r,l\geq0$, see \cite[Theorem 3.22]{CELW}.
\subsection{Monoidal model categories}
We take \cite[Definition 1.1.3]{H} as our definition of a \textit{model category} and \cite[Definition 2.1.17]{H} for that of a \textit{cofibrantly generated model category}.
For a cofibrantly generated model category $\mc{M}$ we denote by $I$ the \textit{generating cofibrations} and $J$ the \textit{generating acyclic cofibrations}.
Diagrammatically cofibrations and fibrations may be shown as $\rightarrowtail$ and $\twoheadrightarrow$, and arrows adorned with $\sim$ are weak equivalences.

For a class of morphisms $K$ we write $K\Inj$ (resp.\ $K\Proj$) for the \textit{$K$-injectives} (resp.\ \textit{$K$-projectives}), i.e.\ those morphisms with the right (resp.\ left) lifting property with respect to all morphisms of $K$.
We also write $I\Cof\coloneqq(I\Inj)\Proj$.
Recall then that in the context of a cofibrantly generated model category that the \textit{cofibrations} $\Cofib$ of $\mc{M}$ are given by $I\Cof$ and the \textit{fibrations} $\Fib$ by $J\Inj$.

We further denote by $K\Cell$ the \textit{$K$-cellular} morphisms that are formed by taking a transfinite composition of iterated pushouts of elements of $K$, \cite[Definition 2.1.9]{H}.

The following definitions can be found in the literature as \cite[Definitions 4.2.1 and 4.2.6]{H}.
By a \textit{monoidal category} we mean symmetric monoidal and we write $\otimes$ for the monoidal product and $\I$ for its unit.
\begin{defi}
  Let $\mc{C}$ be a monoidal category. The \textit{pushout-product} of two morphisms $f\colon A\rightarrow B$ and $g\colon C\rightarrow D$ is denoted $f\boxtimes g$ and given by:
  \begin{equation*}
    f\boxtimes g \colon A\otimes D\coprod_{A\otimes C} B\otimes C \longrightarrow B\otimes D\;.
  \end{equation*}
  For classes of morphisms $S$ and $T$ we also write $S\boxtimes T$ for the class of those morphisms which are pushout-products of a morphism of $S$ by one of $T$.
\end{defi}
\begin{defi}\label{monoidalModelCategory}
  A \textit{monoidal model category} $\mc{M}$ is a model category with a symmetric monoidal structure that is closed, i.e.\ $A\otimes -$ has a right adjoint, and satisfies the following axioms:
  \begin{itemize}
  \item The \textit{pushout-product axiom}: that $f\boxtimes g$ is a cofibration whenever both $f$ and $g$ are, and is further an acyclic cofibration if additionally either $f$ or $g$ is a weak equivalence.
  \item The \textit{unit axiom}: that there is a weak equivalence $Q\I\rightarrow \I$ with $Q\I$ cofibrant such that the induced map $Q\I\otimes X \rightarrow \I\otimes X \rightarrow X$ is a weak equivalence for all cofibrant $X$.
  \end{itemize}
\end{defi}
In fact if the unit axiom is satisfied for some cofibrant replacement of $\I$ then it is satisfied for any cofibrant replacement of the unit, \cite[Lemma 7]{Muro}.

In a cofibrantly generated model category one can verify the the monoidal model category axioms on the set of generating (acyclic) cofibrations.
\begin{lemm}[{\cite[Lemma 2.4.2]{H}}]\label{PushoutProductOnGeneratingCofibrations}
  Let $\mc{M}$ be a cofibrantly generated model category with a closed symmetric monoidal structure.
  If $I\boxtimes I\subseteq I\Cof$, $I\boxtimes J\subseteq J\Cof$ and $J\boxtimes J\subseteq J\Cof$ then $\mc{M}$ satisfies the pushout-product axiom.\qed
\end{lemm}
The major consequence of a model category being a monoidal model category is that its homotopy category has an induced structure of a monoidal category with unit, \cite[Theorem 4.3.2]{H}.
The following is Schwede and Shipley's definition, \cite[Definition 3.3]{SS}, of an axiom yielding extra model category structures of algebras and modules given a monoidal model category.

For a class of morphisms $\mc{G}$ of a category $\mc{M}$ denote by $\mc{G}\otimes\mc{M}$ the class of those morphisms of the form $g\otimes \id_M$ for some $g\in\mc{G}$ and $ M$ an object of $\mc{M}$.
\begin{defi}\label{MonoidAxiom}
  \sloppy A monoidal model category $\mc{M}$ satisfies the \textit{monoid axiom} if every morphism of $\left(\left(\mc{W}\cap\Cofib\right)\otimes\mc{M}\right)\Cell$ is a weak equivalences.
\end{defi}
\begin{lemm}[{\cite[Lemma 3.5 (2)]{SS}}]\label{MonoidAxiomFromJ}
  Let $\mc{M}$ be a cofibrantly generated model category with a closed symmetric monoidal structure.
  If every morphism of $\left(J\otimes \mc{M}\right)\Cell$ is a weak equivalence then the monoid axiom holds.\qed 
\end{lemm}
\begin{theo}[{\cite[Theorem 4.1]{SS}}]\label{SSModelStructures}
  Let $\mc{M}$ be a cofibrantly generated monoidal model category such that every object is small relative to $\mc{M}$ and that $\mc{M}$ satisfies the monoid axiom. let $I$ and $J$ be the generating cofibrations and acyclic cofibrations respectively.
  \begin{enumerate}
  \item If $R$ is monoid in $\mc{M}$ there is a cofibrantly generated model category of left $R$-modules in $\mc{M}$.
  \item If $R$ is a commutative monoid in $\mc{M}$ there is a cofibrantly generated monoidal model category of $R$-modules in $\mc{M}$ satisfying the monoid axiom.
  \item If $R$ is a commutative monoid of $\mc{M}$ there is a cofibrantly generated model category of $R$-algebras in $\mc{M}$.
  \end{enumerate}
  Furthermore the weak equivalences (resp.\ fibrations) are those morphisms that are the weak equivalences (resp.\ fibrations) of $\mc{M}$ after forgetting the monoid or module structure.
  The generating cofibrations and acyclic cofibrations in the first two cases are given by $R\otimes I$ and $R\otimes J$ and in the third case by $T_RI$ and $T_RJ$ where $T_R$ is the free $T$-algebra functor.\qed
\end{theo}
Schwede and Shipley note that their result does not obtain a model structure on the commutative $R$-algebra objects and in fact give an example in \cite[Remark 4.5]{SS} where there is no model category structure whose weak equivalences and fibrations are those of $\mc{M}$ after forgetting the monoid structure.
\subsection{Muro's results}\label{MurosResults}
In \cite{Muro} Muro considers a cofibrantly generated monoidal model category $\mc{M}$ and defines a new cofibrantly generated monoidal model category $\wt{\mc{M}}$ with the same underlying category, monoidal structure and weak equivalences in which the unit of the tensor product is now cofibrant.
In addition there is a \textit{monoidal Quillen equivalence} $\inadj{\id}{\mc{M}}{\wt{\mc{M}}}{\id}$ and $\wt{\mc{M}}$ satisfies similar model category properties as $\mc{M}$, outlined in \cref{MurosTheorem}.
New generating (acyclic) cofibrations are added to those of $\mc{M}$ in the construction.

We recall the relevant definitions and results from \cite{Muro}.

\begin{defi}\label{veryStrongUnitAxiom}
  A monoidal model category $\mc{M}$ satisfies the \textit{very strong unit axiom} if it satisfies the unit axiom for any $X$, i.e.\ we drop the requirement that $X$ be cofibrant.
\end{defi}
Muro makes the following definitions as part of \cite[Theorem 3]{Muro}.
\begin{defi}\label{MurosFactorisation}
  Let $q\colon Q{\I}\overset{\sim}{\rightarrow}\I$ be a cofibrant replacement of the unit satisfying the very strong unit axiom.
  Then write $Q\I\coprod\I\overset{j}{\rightarrowtail} C \underset{\sim}{\overset{p}{\rightarrow}} \I$ for a factorisation of $(q,id)\colon Q\I\coprod\I \rightarrow\I$ into a cofibration followed by a weak equivalence.
  Also let $i_1\colon Q\I\rightarrow Q\I\coprod\I$ be the inclusion of the first factor.
\end{defi}
\begin{defi}
  Let $\mc{M}$ be a cofibrantly generated monoidal model category. Define $\wt{I}$ and $\wt{J}$ to be the sets of morphisms given by:
  \begin{align*}
    \wt{I}&\coloneqq I\cup\left\{\emptyset\rightarrow\I\right\},\\
    \wt{J}&\coloneqq J\cup\left\{j\circ i_1\colon Q\I\rightarrow C\right\}
  \end{align*}
  for $\emptyset$ the initial object of $\mc{M}$, $i_1$ and $j$ as in \cref{MurosFactorisation} and $I$ and $J$ the generating (acyclic) cofibrations of $\mc{M}$.
\end{defi}
\begin{theo}[{\cite[Theorem 3]{Muro}}]\label{MurosTheorem}
  Let $\mc{M}$ be a cofibrantly generated monoidal model category with $q\colon Q\I\overset{\sim}{\rightarrow}\I$ satisfying the very strong unit axiom.
  Then if all objects of $\mc{M}$ are small relative to $\mc{M}$ there is a cofibrantly generated monoidal model category $\wt{\mc{M}}$ with the same underlying category, weak equivalences and monoidal structure as $\mc{M}$ with sets of generating cofibrations $\wt{I}$ and generating acyclic cofibrations $\wt{J}$.
  Furthermore if $\mc{M}$ is right (resp.\ left) proper then so too is $\wt{\mc{M}}$. If $\mc{M}$ is symmetric monoidal satisfying the monoid axiom so too does $\wt{\mc{M}}$.\qed
\end{theo}

\subsection{Existing model structures}\label{CELWresults}
We recall the model structures on filtered chain complexes of \cite{CELW}.
For each $r\geq 0$ they give two model structures denoted $\fChr$ and $\fChrp$ (which agree for $r=0$).
They define the following sets of morphisms:
\begin{defi}
  For $r\geq 0$ the sets $I_r$ and $J_r$ of morphisms of $\fCh$ are given by:
  \begin{align*}
    I_r&\coloneqq \left\{\Z_{r+1}(p,n)\longrightarrow
         \B_{r+1}(p,n)\right\}_{p,n\in\Zbb}\;,\\
    J_r&\coloneqq \left\{0\longrightarrow \Z_r(p,n)
         \right\}_{p,n\in\Zbb}\;.
  \end{align*}
\end{defi}
\begin{theo}[{\cite[Theorem 3.14]{CELW}}]
  \label{filtChainsr}
  For each $r\geq 0$ there is a right proper cofibrantly generated model structure on $\fCh$ whose weak equivalences are $\Er$ with generating cofibrations $I_r$ and generating acyclic cofibrations $J_r$.
  The fibrations are those morphisms that are $Z_r$-bidegreewise surjective.\qed
\end{theo}
\begin{theo}[{\cite[Theorem 3.16]{CELW}}]
  \label{filtChainsrp}
    For each $r\geq 0$ there is a right proper cofibrantly generated model structure on $\fCh$ whose weak equivalences are $\Er$ with generating cofibrations $I_r\cup\left(\bigcup_{k=0}^{r-1}J_k\right)$ and generating acyclic cofibrations $\bigcup_{k=0}^rJ_k$.
  The fibrations are those morphisms that are $Z_k$-bidegreewise surjective for all $0\leq k\leq r$.\qed
\end{theo}
Via \cite[Lemma 2.8]{CELW} they also show the fibration condition of $Z_k$-bidegreewise surjectivity for all $0\leq k\leq r$ is equivalent to $E_k$-bidegreewise surjectivity for the same $k$.

We state the propositions shown to prove existence of these model categories; we use the same lemmas to show existence of further model structures.
\begin{prop}[{\cite[Proposition 3.12]{CELW}}]
  \label{IrInj}
  For each $r\geq0$ we have $I_r\Inj=J_r\Inj\cap\Er$.\qed
\end{prop}
\begin{prop}[{\cite[Proposition 3.13]{CELW}}]
  \label{JkCof}
  For each $r\geq0$ and $0\leq k\leq r$ we have $J_k\Cof\subseteq\Er$.\qed
\end{prop}
The model structures of \cref{filtChainsr,filtChainsrp} now follow from these results and the following result due to Kan, \cite[Theorem 2.1.19]{H}.
\begin{theo}\label{KanModCat}
  Let $\mc{C}$ be a category closed under all small (co)limits, with $\W$ a subclass of the morphisms, and $I$ and $J$ subsets of the morphisms.
  Then $I$ and $J$ determine a cofibrantly generated model category with weak equivalences $\W$ if and only if:
  \begin{enumerate}
  \item $\W$ satisfies the two-out-of three property,
  \item the domains of $I$ (resp.\ $J$) are small relative to $I\Cell$ (resp.\ $J\Cell$),
  \item $J\Cell\subseteq\W\cap I\Cof$,
  \item $I\Inj\subseteq\W\cap J\Inj$, and
  \item either $\W\cap I\Cof\subseteq J\Cof$ or $\W\cap J\Inj\subseteq I\Inj$.\qed
  \end{enumerate}
\end{theo}

\subsection{Projective model structure on chain complexes}\label{Projective}
We recall some facts about the projective model structure on (unbounded cohomologically graded) chain complexes.
These are known from e.g.\ \cite{H}.
Write $S^n$ for the chain complex with a copy of $R$ in degree $n$ and $0$ otherwise, and similarly $D^n$ for the complex with a copy of $R$ in degrees $n$ and $n-1$ with an identity differential between them.
The projective model structure on chain complexes has weak equivalences the homology isomorphisms (quasi-isomorphisms) and generating cofibrations $I$ and $J$ given by:
\begin{align*}
  I&\coloneqq \left\{S^{n-1}\rightarrow D^n\right\}_{n\in\Zbb}\,,\\
  J&\coloneqq \left\{0\rightarrow D^n\right\}_{n\in\Zbb}\,.
\end{align*}
Fibrations are then those morphisms that are degreewise surjective.

Cofibrations are not quite as easily described.
For a bounded below degree-wise projective objects $X$ one can show via an inductive argument that $X$ is cofibrant.
In general a cofibrant object $X$ is such that $X$ is degree-wise projective and lifts exist in any lifting problem of the form:
\begin{equation*}
  \begin{tikzcd}
    &C(K)\arrow[d,"\pi", two heads, "\sim"']\\
    X\arrow[r]\arrow[ur, dashed]&\Sigma K
  \end{tikzcd}
\end{equation*}
whenever $K$ is an acyclic object and $C(K)\rightarrow \Sigma K$ is the usual fibration from the cone onto the suspension of $K$.

Such an $X$ satisfying these conditions can be shown to be cofibrant as follows: given a lifting problem against an acyclic fibration $E\overset{\sim}{\twoheadrightarrow}B$ one can find a lift of graded modules using degree-wise projectivity of $X$.
This lift $g$ is not compatible with the differential, however $gd-dg\colon X\rightarrow \Sigma E$ is a morphism of chain complexes and moreover the image is in $\Sigma K$ where $K = \ker \left(E\rightarrow B\right)$ and note $K$ is acyclic.
One can then use the second condition, lifts of $X$ against $C(K)\rightarrow K$, to correct for the discrepancy with the differential yielding a lift of $X$ against $E\rightarrow B$ as chain complexes.
Cofibrations can then be shown to be degreewise-split monomorphisms with cofibrant cokernel.

We provide this classification of the cofibrations here as it will be used as a basis for giving a subclass of the cofibrations in $\fChr$.


%% file: tex/construction.tex
\section{Construction of new model structures}\label{construction}
We now show existence of further model structures ``lying in between those'' of $\fChr$ and $\fChrp$.
The set $S$ for the remainder of the paper shall denote a subset of $S\subseteq\{0,1,2,\ldots,r\}$ which must include $r$.
\begin{defi}\label{ISandJS}
    The sets $I_S$ and $J_S$ of morphisms of $\fCh$ are given by:
  \begin{align*}
    I_S&\coloneqq I_r\cup\bigcup_{s\in S}J_s\,,\\
    J_S&\coloneqq \bigcup_{s\in S}J_s\,.
  \end{align*}
\end{defi}
Restricting to either $S=\{r\}$ or $S=\{0,1,2,\ldots,r\}$ gives the generating cofibrations and generating acyclic cofibrations for the model structures of \cref{filtChainsr,filtChainsrp} respectively.
\begin{prop}
  \label{ISInj}
  For each $S$ we have $I_S\Inj=J_S\Inj\cap\Er$.
\end{prop}
\begin{proof}
  In the following the third equality follows from \cref{IrInj}.
  \begin{align*}
    I_S\Inj&=\left(I_r\cup\bigcup_{s\in S}J_s\right)\Inj\\
           &=I_r\Inj\cap\bigcap_{s\in S}J_s\Inj\\
           &= \left(J_r\Inj\cap\Er\right)\cap\bigcap_{s\in S}J_s\Inj\\
           &=J_S\Inj\cap\Er\,.\qedhere
  \end{align*}
\end{proof}
\begin{theo}
  \label{filtChainsS}
  For each $S$ there is a right proper cofibrantly generated model structure on $\fCh$ whose weak equivalences are $\Er$ with generating cofibrations $I_S$ and generating acyclic cofibrations $J_S$.
  The fibrations are those morphisms that are $Z_s$-bidegreewise surjective for each $s\in S$.
\end{theo}
\begin{proof}
  Existence of these model structures follows from \cref{KanModCat} and \cref{JkCof,ISInj}.
  They are right proper since every object is fibrant.
\end{proof}

Given such a model structure on $\fCh$ determined by the set $S$ we shall denote the model structure by $\fChS$.
$S$-(acyclic) cofibrations, $S$-(acyclic) fibrations and $S$-cofibrant objects will denote the (acyclic) cofibrations, (acyclic) fibrations and cofibrant objects in $\fChS$.
When $S=\left\{r\right\}$, we replace $S$ in these notations with $r$.


%% file: tex/cofibrations.tex
\section{Cofibrations}\label{cofibrations}
Recall in \cref{Projective} we gave an outline for the proof of the classification of cofibrations in the projective model structure on chain complexes.
We follow a similar method here to give a subclass of cofibrations in the $r$-model structure $\fChr$.
Our subclass of cofibrant objects and cofibrations will have conditions analogous to those in chain complexes, including projectivity and lifts against morphisms similar to $C(K)\rightarrow K$ for $K$ acyclic, in addition to new ones taking account of the $r$-homotopy nature of the model structure.
\subsection{A subclass of cofibrant objects}
The known properties of a cofibrant $A\in\fChr$ can be summarised in the following lemma.
\begin{prop}\label{cofibrantConditions}
  Let $A$ be a cofibrant object in $\fChr$. Then $A$ satisfies the following conditions:
  \begin{enumerate}
  \item $A^n$ and $A^n/F_pA^n$ are projective for all $p,n\in\Zbb$,
  \item the filtration on $A$ is exhaustive,
  \item for $a\in F_pA^n$ we have $da\in F_{p-r}A^{n+1}$, and
  \item for any morphism $A\rightarrow \rSusp K$ of $A$ into (the $r$-suspension of) an $r$-acyclic $K$ there is a lift against $C_r(K)\rightarrow \rSusp K$.
  \end{enumerate}
\end{prop}

The next few lemmas demonstrate these properties.
\begin{defi}
  Let $\pi\colon N\twoheadrightarrow M$ be a surjection of $R$-modules.
  We define $\sigma_s^{p,p+n}$ by:
  \begin{equation*}
    \sigma_s^{p,p+n}\colon \Z_s(p+s,n-1)\otimes N\longrightarrow
    \Z_s(p+s,n-1)\otimes M
  \end{equation*}
  to be the tensor product of the identity on the $s$-cycle with $\pi$, where we interpret $\pi$ as a morphism of filtered chain complexes concentrated in cohomological degree 0 and of pure filtration degree $0$.
\end{defi}
\begin{lemm}
  The morphisms $\sigma_r^{p,p+n}$ are $r$-acyclic fibrations.\qed
\end{lemm}
\begin{lemm}\label{AFpAProjective}
  Let $A\in\fChr$ be cofibrant. Then $A^n/F_pA^n$ is projective for each $p,n\in\Zbb$.
\end{lemm}
\begin{proof}
  A morphism $A\rightarrow \Z_0(p,n-1)\otimes N$ is equivalently a morphism $A^n/F_{p-1}A^n\rightarrow N$.
  Since $\sigma_0^{p,p+n}$ is an $r$-acyclic fibration and $A$ is cofibrant there is then a lift of $A^n/F_{p-1}A^n$ against $N\twoheadrightarrow M$, hence $A^n/F_{p-1}A^n$ is projective.
\end{proof}
One can also take $p=-\infty$ in the previous taking this to mean the $R$-module $N$ is in all filtration degrees in which case one obtains the following lemma.
\begin{lemm}\label{AProjective}
  Let $A\in\fChr$ be cofibrant. Then $A^n$ is projective for all $n\in\Zbb$.\qed
\end{lemm}
\begin{lemm}\label{CofibrantExhaustive}
  Let $A\in\fChr$ be cofibrant. Then the filtration on $A$ is exhaustive.
\end{lemm}
\begin{proof}
  Denote by $\bar{A}$ the union of the filtered pieces $\cup_pF_pA$. It is a filtered chain complex with the obvious filtration and there's an inclusion $i\colon \bar{A}\hookrightarrow A$.
  In fact the morphism $i$ is an $r$-acyclic fibration: all conditions on fibrations and weak equivalences are given by the data of finite filtration indexing and the $r$-cycles and $r$-boundaries are unaffected by the operation $A\mapsto \bar{A}$.

  The lifting problem:
  \begin{equation*}
    \begin{tikzcd}
      &\bar{A}\arrow[d,"i", "\sim"', two heads]\\
      A\arrow[ur,dashed]\arrow[r,"id"]&A
    \end{tikzcd}
  \end{equation*}
  must then have a solution since $A$ is cofibrant and $i$ is an $r$-acyclic fibration.
  This shows any element of $A$ lives in finite filtration degree so that $A$ is exhaustive.
\end{proof}
\begin{lemm}\label{APushoutWithGenCofib}
  Let $A\in\fCh$ and $f\colon\Z_{r+1}(p,n)\rightarrow A$ be an $(r+1)$-cycle of $A$. The pushout of $f$ along $i_{r+1}\colon\Z_{r+1}(p,n)\rightarrow\B_{r+1}(p,n)$ is given by the filtered chain complex we denote $A'$ with underlying filtered graded module given by $A'\coloneqq A\oplus R_{(p+r)}^{n-1}\{\gamma\}\oplus R_{(p-1)}^n\{\alpha\}$. The differential $d^{A'}$ is given on elements of $A\subset A'$ by $d^A$ and on $\gamma$ by $d\gamma=a-\alpha$ so that $d\alpha=da$.\qed
\end{lemm}
We consider the first $r$-pages of the spectral sequence associated to a cofibrant $A\in\fChS$.
The following is shown by a simple retract argument.
\begin{lemm}\label{CofibrantRetract}
  Let $A,B\in\fCh$ be such that $A$ is a retract of $B$ and the $k$-page differential of $B$ is $0$. Then the $k$-page differential of $A$ is also $0$.\qed
\end{lemm}
\begin{lemm}\label{APushoutGenCofibHasKPageZero}
  Let $A\in\fCh$ be such that the $k$-page differential of $A$ is $0$ with $k<r$ and $k\notin S$. Then the pushout of $A$ along a morphism of $I_S$ also has $k$-page differential being $0$.
\end{lemm}
\begin{proof}
  For those morphisms of $I_S$ of the form $0\rightarrow\Z_s(p,n)$ the statement follows since the pushout is $A\rightarrow A\oplus\Z_s(p,n)$.
  Consider then the pushout of an $A\in\fCh$ along a generating cofibration of the form $\Z_{r+1}(p,n)\rightarrow\B_{r+1}(p,n)$.
  The pushout is given by the $A'$ of \cref{APushoutWithGenCofib}.
  By commutativity of
  \begin{equation*}
    \begin{tikzcd}
      E_k^{p,p+n}(A)\arrow[r,"d_k^{A}"]\arrow[d]&E_k^{p-k,p-k+n+1}(A)\arrow[d]\\
      E_k^{p,p+n}(A')\arrow[r,"d_k^{A'}"]&E_k^{p-k,p-k+n+1}(A')
    \end{tikzcd}
  \end{equation*}
  any element of $E_k(A')$ represented by a $k$-cycle coming from $A$ will still be $0$ under $d_k^{A'}$.
  The new $k$-cycles of the pushout are $\alpha$ and $\gamma$.
  But note that $d\alpha\in Z_{k-1}^{p-k-1,p-k-1+n+1}(A')\subset B_k^{p-k,p-k+n+1}(A')$ hence $d^{A'}_k([\alpha])=0$ and similarly for $\gamma$.
\end{proof}
\begin{lemm}\label{CellularKPage}
  Let $A\in\fCh$ be such that the $k$-page differential of $A$ is $0$ with $k<r$ and $k\notin S$. Then the $k$-page differential of the codomain $B$ of any relative $I_S\Cell$ complex whose domain is $A$ is also $0$.
\end{lemm}
\begin{proof}
  Recall an $I_S\Cell$ with domain $A$ is a transfinite composition of pushouts of the generating cofibrations.
  Denote the stages of this composition by $A_\bullet$.
  We prove the codomain of such an $I_S\Cell$ has $k$-page differential also $0$ when the $k$-page differential of $A$ is.
  The base case is provided by the hypothesis of the lemma, the successor ordinal case was proven in \cref{APushoutGenCofibHasKPageZero} so it remains to prove the limit ordinal case.

  Let $\lambda$ be a limit ordinal of the transfinite composition and
  let $[x]$ be an element of the $k$-page of the $A_\lambda$ represented by a $k$-cycle $x$ of $A_\lambda$.
  In the transfinite composition there are the elements $y$ and $z$ in some stages of the composition given by ordinals $\eta$ and $\theta$ less than $\lambda$ whose image in $A_\lambda$ is $x$ and $dx$ respectively and such that the filtration degrees of $y$ and $x$ agree as do those of $z$ and $dx$.
  Taking the image of $y$ in the stage of the composition indexed by the larger of $\eta$ and $\theta$ gives another $k$-cycle whose image in $A_\lambda$ is $x$ and such that, by induction, $d_k[y]=0$.
  We also then have $d_k[y]=0$ in $A_\lambda$ which by transfinite induction proves the claim.
\end{proof}
\begin{coro}\label{cofibrantkPageDifferentialIsZero}
  Let $A\in\fChS$ be cofibrant. Then the $k$-page differential of $A$ with $k<r$ and $k\notin S$ is $0$.
\end{coro}
\begin{proof}
  Any cofibrant object is the retract of a cellular one. The cellular objects have $k$-page differential being $0$ for $k<r$ and $k\notin S$ by \cref{CellularKPage} and the result for all cofibrant objects follows from \cref{CofibrantRetract}.
\end{proof}
This then gives another description of how, for a fixed $r$, the model structures $\fChS$ differ.
They can be seen to differ on their cofibrant objects where for $k<r$ the $k$-page differential is necessarily $0$ if $k\notin S$.
There is no analogous result for $k>r$ by definition of the weak equivalences.
A consequence of this corollary is the following corollary for the $\fChr$ model structures.
\begin{coro}\label{rCofibrantPages}
  Let $A\in\fChr$ be cofibrant. Then the $k$-page differential is $0$ for $0\leq k<r$. In particular we have isomorphisms:
  \begin{equation*}
    E_0^{p,p+n}(A)\cong E_1^{p,p+n}(A)\cong \ldots \cong E_r^{p,p+n}(A)\;.\pushQED{\qed}\qedhere
  \end{equation*}
\end{coro}
Results of this form have appeared before in the thesis of Cirici, \cite[Lemma 4.3.14]{Cirici} for their \textit{$E_r$-cofibrant dgas} constructed by iteratively tensoring the base field with exterior algebras.
\begin{proof}[Proof of \cref{cofibrantConditions}]
  Condition 1 was shown in \cref{AFpAProjective,AProjective}, condition 2 in \cref{CofibrantExhaustive}, condition 3 follows from \cref{rCofibrantPages} and condition 4 follows since $C_r(K)\rightarrow \rSusp K$ is an $r$-acyclic fibration.
\end{proof}
We now wish to give a partial converse to this lemma, i.e.\ conditions on a filtered chain complex that ensure cofibrancy in $\fChr$.
The same conditions 1--4 feature in addition to an extra boundedness assumption on the filtration.
This additional assumption is not a necessary condition, \cref{lemm:QrIIsCofibrant} provides a counterexample.
\begin{prop}\label{subclassOfCofibrantObjects}
  Let $A\in\fCh$ be such that it satisfies the conditions of \cref{cofibrantConditions} listed again here:
    \begin{enumerate}
  \item $A^n$ and $A^n/F_pA^n$ is projective for all $p,n\in\Zbb$,
  \item the filtration on $A$ is exhaustive,
  \item for $a\in F_pA^n$ we have $da\in F_{p-r}A^{n+1}$,
  \item for any morphism $A\rightarrow \rSusp K$ of $A$ into (the shift of) an $r$-acyclic $K$ there is a lift against $C_r(K)\rightarrow \rSusp K$,
  \end{enumerate}
  in addition to a fifth assumption:
  \begin{enumerate}
    \setcounter{enumi}{4}
  \item for any cohomological degree $n$ there exists a $p(n)$ such that $F_{p(n)}A^n=0$,
  \end{enumerate}
  then $A$ is cofibrant in $\fChr$.
\end{prop}
\begin{proof}
  Given conditions 1 and 5 one can inductively split any $A^n$ into its graded pieces and condition 2 ensures the final isomorphism in the limit of the following:
  \begin{align*}
    A^n&\cong \frac{F_{p(n)+1}A^n}{F_{p(n)}A^n}\oplus\frac{A^n}{F_{p(n)+1}A^n}\\
       &\cong \frac{F_{p(n)+1}A^n}{F_{p(n)}A^n}\oplus\frac{F_{p(n)+2}A^n}{F_{p(n)+1}A^n}\oplus\frac{A^n}{F_{p(n)+2}A^n}\\
       &\cong\ldots\\
       &\cong\bigoplus_{p(n)\leq p}\frac{F_{p+1}A^n}{F_pA^n}\;.
  \end{align*}
  We drop the $p(n)$ from the sum as we no longer have need for it.
  The induced filtration is the obvious one $F_pA^n\cong\bigoplus_{q\leq p}F_qA^n/F_{q-1}A^n$.
  Consider a summand of $F_pA^n/F_{p-1}A^n$ of this splitting of $A^n$.
  Condition 3 ensures that its image under the differential is in the subobject $\bigoplus_{q\leq p-r}F_qA^{n+1}/F_{q-1}A^{n+1}$ of the splitting for $A^{n+1}$.
  We then have that any element of ${F_pA^n}/{F_{p-1}A^n}$ is in $Z_r^{p,p+n}(A)$.

  Consider then a lifting problem of the form:
  \begin{equation*}
    \begin{tikzcd}
      &E\arrow[d,two heads,"f","\sim"']\\
      A\arrow[r]\arrow[ur,dashed]&B
    \end{tikzcd}
  \end{equation*}
  where $f$ is an $r$-acyclic fibration.
  For a splitting of $A^n$ for each $n$ one can then lift each graded piece, irrespective of compatibility with the differential, using $r$-cycle surjectivity of $f$ and projectivity of the graded pieces giving a lift of graded filtered modules.
  Finally one uses condition 4 to correct for the discrepancy with the differential in the same way as in the projective model structure on chain complexes.
\end{proof}
\subsection{A subclass of cofibrations}
Using the description of a subclass of cofibrant objects just obtained we give a similar one for a subclass of cofibrations.
Recall for analogy in the projective model structure of chain complexes these are the degree-wise split monomorphisms with cofibrant cokernel.
\begin{defi}
  Let $f\colon A\rightarrow B$ be a morphism in $\fCh$. We say that $f$ is \textit{strict} if whenever $f(a)\in F_pB$ we have $a\in F_pA$.
\end{defi}
Note the computation of cokernels of strict morphisms is done degreewise with respect to both filtration and cohomological degree by the construction of colimits given in \cref{filteredChainsLimits}.
\begin{lemm}
  Let $f\colon A\rightarrow B$ be a cofibration in $\fChS$, then $f$ is a strict morphism.
\end{lemm}
\begin{proof}
  Let $N\coloneqq A^n/F_{p}A^n$ considered as a filtered chain complex concentrated in cohomological degree $0$ and of pure filtration degree $0$.
  Note there is a morphism $A\rightarrow N\otimes\Z_0(p+1,n-1)$ which in cohomological degree $n$ is the quotient morphism $A^n\rightarrow A^n/F_{p}A^n$ and  in degree $n-1$ is this quotient precomposed by the differential $d\colon A^{n-1}\rightarrow A^n$.
  Since $N\otimes\Z_0(p+1,n-1)$ is $0$-acyclic, and hence $r$-acyclic as well as being $S$-fibrant there is a lift in the diagram:
  \begin{equation*}
    \begin{tikzcd}
      A\arrow[d,tail,"f"]\arrow[r]&N\otimes\Z_0(p+1,n-1)\arrow[d,"\sim",two heads]\\
      B\arrow[r]\arrow[ur,dotted,"h"]&0
    \end{tikzcd}
  \end{equation*}
  and given an element $f(a)\in F_pB^n$, its image under $h$ must be $0$ which shows $a\in F_pA^n$ and so $f$ is a strict morphism.
\end{proof}
Replacing $N$ with $N\coloneqq A^n$ in the preceding proof also demonstrates the following lemma.
\begin{lemm}
  Let $f\colon A\rightarrow B$ be a cofibration in $\fChS$, then $f$ is an inclusion.\qed 
\end{lemm}
\begin{defi}\label{defi:twistedDirectSum}
  let $A,C\in\fCh$. A \textit{twisted direct sum} of $A$ and $C$, denoted $A\toplus C$, is a filtered chain complex with underlying filtered graded module given by $A\oplus C$ and with differential given by:
  \begin{equation*}
    d^{A\toplus C}\coloneqq \begin{pmatrix}d^A&\tau\\0&d^C\end{pmatrix}\colon
    A\toplus C\rightarrow A\toplus C
  \end{equation*}
  for some degree $1$ morphism $\tau$ (we call the \textit{twist morphism}).
  In particular $d^A\tau+\tau d^C=0$.
\end{defi}
\begin{defi}\label{rSupressive}
  For $A\toplus C$ a twisted direct sum of filtered chain complexes we will say the twist morphism $\tau$ is \textit{$r$-suppressive} if for any $p,n\in\Zbb$ and $c\in F_pC^n$ then $\tau(c)\in F_{p-r}A^{n+1}$.
\end{defi}
Since pushouts of cofibrations are cofibrations, the cokernel $C$ of an $r$-cofibration $A\rightarrow B$ must be $r$-cofibrant, in particular $C^n$ is projective by \cref{cofibrantConditions} so by a standard splitting argument we have the following lemma.
\begin{lemm}\label{cofibrationsAreInclusionsOfTwistedSums}
  Let $f\colon A\rightarrow B$ be a cofibration of $\fChr$ and $C$ its cokernel, then $f$ is isomorphic to a morphism of filtered chain complexes of the form $i\colon A\rightarrow A\toplus C$ which is the inclusion onto the first summand (with $C$ cofibrant).\qed
\end{lemm}
\begin{defi}
  Let $f\colon A\rightarrow B$ be a strict inclusion isomorphic to the inclusion $i\colon A\rightarrow A\toplus C$ of $A$ into the twisted direct sum of $A$ with its cokernel.
  We say $f$ is an \textit{$r$-suppressive inclusion} if the twist map $\tau$ is $r$-suppressive.
\end{defi}
We can now state our result giving a subclass of cofibrations which is a partial converse to \cref{cofibrationsAreInclusionsOfTwistedSums}.
The limitation is similar to that for the subclass of cofibrant objects, i.e.\ we will require the same boundedness condition on the the cokernel of our cofibrations.
\begin{prop}\label{subclassOfCofibrations}
  Let $f\colon A\rightarrow B$ be an $r$-suppressive inclusion with cofibrant cokernel $C$.
  Suppose further that $C$ satisfies condition 5 of \cref{subclassOfCofibrantObjects}, then $f$ is an $r$-cofibration.
\end{prop}
\begin{proof}
  The proof is similar to that for \cref{subclassOfCofibrantObjects}.
  Given an $r$-acyclic morphism $\pi\colon Y\rightarrow B$ and a lifting problem of the form :
  \begin{equation*}
    \begin{tikzcd}
      A\arrow[d,"i"]\arrow[r]&Y\arrow[d,"\sim"',two heads, "\pi"]\\
      A\toplus C\arrow[r]&X
    \end{tikzcd}
  \end{equation*}
  one first writes $C$ as a direct sum of its graded pieces which are each projective and consist of $r$-cycles so can be lifted on $C$ by the $r$-cycle surjectivity of $\pi$ as a lift of filtered graded modules $h$ so not necessarily compatibly with the differential.
  The lift on the $A$ summand is already determined by the morphism $A\rightarrow Y$.
  This incompatibility can then be rectified by considering the difference $hd^{A\toplus C}-d^Yh\colon C\rightarrow \rSusp K$ where $K$ is the $r$-acyclic kernel of $\pi\colon Y\rightarrow X$ and using existence of a lift in
  \begin{equation*}
    \begin{tikzcd}
      &C_r(\rSusp K)\arrow[d,"\pi_r", two heads, "\sim"']\\
      C\arrow[r]\arrow[ur, dotted]&\rSusp K
    \end{tikzcd}
  \end{equation*}
  to provide a homotopical correction for the discrepancy of the lift $h$ with the differential.
\end{proof}
\subsection{Remarks on cofibrations}
We have already remarked that \cref{subclassOfCofibrations} only provides a subclass of the cofibrations of $\fChr$, indeed the $Q_r\I$ of \cref{rCofibrantUnit} will give an example of a cofibrant object not satisfying the boundedness condition on the filtration.
We list a few other remarks concerning the authors expectation of properties of cofibrations.

\cref{cofibrationsAreInclusionsOfTwistedSums} exhibited an $r$-cofibration as the inclusion $A\rightarrow A\toplus C$ into a twisted direct sum.
The author suspects that for any cofibration of $\fChr$ written in this form that the twist morphism $\tau$ must be $r$-suppressive (\cref{rSupressive}).

Much of the focus of this section was on cofibrations of $\fChr$ specifically as this is sufficient for our later proof of a monoidal model structure on $\fChS$.
Note that for cofibrant objects of $\fChS$ conditions 1, 2 and 4 still hold in \cref{cofibrantConditions} however a weaker form of condition 3 is true: for $A\in\fChS$ cofibrant and $a\in F_pA^n$ we have $da\in F_{p-s_0}A^n$ where $s_0=\min S$.
The filtered chain complex $\Z_{s_0}(0,0)$ is such a cofibrant object of $\fChS$.
\subsection{Shift-d\'ecalage adjunction}
Recall the shift-d\'ecalage adjunction of \cref{shiftDecalageAdjunction} on $\fCh$.
The functors satisfy the identity $\Dec\circ\Shift = \id$ however in general $\Shift\circ\Dec \neq \id$, see \cite[Lemma 3.20]{CELW}, so they are not inverse to each other on $\fCh$.
Indeed $\Z_0(p,n)\ncong\Z_1(p+1,n)$ however $\Dec(\Z_0(p,n))\cong\Dec(\Z_1(p+1,n))$.
We will show that when restricted to the subcategories of cofibrant objects in $S$ and $S+1$ that these functors become inverse to each other.
A similar result has appeared in \cite{Cirici}.
Cirici defines a notion of an \textit{$E_r$-cofibrant dga}, \cite[Definition 4.3.14]{Cirici}, and showed in \cite[Lemma 4.3.16]{Cirici} that when restricted to the subcategory of these objects, $E_r\mathdash\mathrm{cof}$, there are inverse functors:
\begin{equation*}
  \inadj{\Shift}{E_r\mathdash\mathrm{cof}}{E_{r+1}\mathdash\mathrm{cof}}{\Dec}\,.
\end{equation*}
\begin{defi}\label{suppressiveObject}
  Let $A\in\fCh$. If $A$ has the property that for all $p,n\in\Zbb$ and $a\in F_pA^n$ then $da\in F_{p-k}A^{n+1}$ we say that $A$ is \textit{$k$-suppressive}.
  We denote by $\Supp{k}$ the full subcategory of $\fCh$ whose objects are $k$-suppressive.
\end{defi}
Equivalently \cref{suppressiveObject} says that every element of $A$ is a $k$-cycle.
We first show that $\Shift$ and $\Dec$ are inverse to each other on categories of suppressive objects.
\begin{lemm}\label{suppressiveEquivalence}
  There is an equivalence of categories:
  \begin{equation*}
    \adj{\Shift^l}{\Supp{k}}{\Supp{k+l}}{\Dec^l}
  \end{equation*}
  and furthermore they are inverse to each other.
\end{lemm}
\begin{proof}
  Given a $k$-suppressive object $A$ we first show that $\Shift^lA$ is $(k+l)$-suppressive.
  Let $a\in F_pA^n$ so that $da\in F_{p-k}A^{n+1}$, these elements in $\Shift^lA$ are in filtration degrees $a\in F_pA^n=F_{p-ln}\Shift^lA^n$ and $da\in F_{p-k}A^{n+1}=F_{p-k-l(n+1)}A^{n+1}=F_{p-ln-(k+l)}A^{n+1}$ so that every element of $\Shift^l A$ is a $(k+l)$-cycle and hence $(k+l)$-suppressive.
  We already know that $\Dec^l\circ\Shift^l=\id$ for any object of $\fCh$.

  It remains to show for a $(k+l)$-suppressive object $B$ that $\Dec^l B$ is $k$-suppressive and that $\Shift^l\circ\Dec^l(B)=B$.

  Let $B$ be $(k+l)$-suppressive so that for $b\in F_pB^n$ we have $db\in F_{p-(k+l)}B^{n+1}$.
  We then have:
  \begin{align*}
    b\in F_pB^n & = Z_l^{p,p+n}(B)\\
                & = F_{p+ln}(\Dec^lB)^n
  \end{align*}
  where the first equality follows since every element of $B$ is a $(k+l)$-cycle and so in particular an $l$-cycle.
  Similarly we have:
  \begin{align*}
    db\in F_{p-(k+l)}B^{n+1}&=Z_l^{p-(k+l), p-(k+l)+n+1}(B)\\
                            & =F_{p-(k+l)+l(n+1)}(\Dec^lB)^{n+1}\\
                            & =F_{p+ln-k}(\Dec^lB)^n
  \end{align*}
  and together the filtration degrees of the elements $b$ and $db$ in $\Dec^lB$ show that $b$ is a $k$-cycle so that $\Dec^l$ sends $(k+l)$-suppressive objects to $k$-suppressive objects.
  Lastly we check for $B$ a $(k+l)$-suppressive object that $S^l\circ\Dec^l (B) = B$:
  \begin{align*}
    F_p(S^l\Dec^lB)^n&=F_{p+ln}(\Dec^lB)^n\\
                     &=Z_l^{p+ln-ln,p+ln-ln+n}B^n\\
                     &=Z_l^{p,p+n}B^n\\
                     &=F_pB^n
  \end{align*}
  where the last equality follows since every element of $B$ is a $(k+l)$-cycle so in particular an $l$-cycle.
\end{proof}
\begin{defi}
  For $S$ a subset of $\mathbb{N}$ and $l\in\mathbb{N}$ we denote by $S+k$ the set $\{s+k\;|\; s\in S\}$.
\end{defi}
Recall from \cref{cofibrantConditions} that every cofibrant object of $\fChS$ was $s_0$-suppressive where $s_0=\min S$ and that in \cite[Theorem 3.22]{CELW} the authors proved Quillen equivalences $\fChr\Quill\left(\fCh\right)_{r+l}$ for all $r,l\geq0$.
This theorem is easily seen to generalise to our setting with $S$-model structures so that we have Quillen equivalences $\fChS\Quill\left(\fCh\right)_{S+l}$ for all finite subsets $S\subset\mathbb{N}$ and $l\in\mathbb{N}$ which are induced by shift-d\'ecalage adjunctions.
In particular the shift functor $\Shift^l$ sends an $S$-cofibrant object to an $(S+l)$-cofibrant object.
\begin{lemm}\label{shiftPreservesAcyclicFibrations}
  Let $\pi$ be a morphism of filtered chain complexes.
  If $\pi$ is $Z_s$-bidegreewise surjective then $\Shift^l\pi$ is $(s+l)$-bidegreewise surjective.
  If $\pi$ is an $r$-quasi-isomorphism then $\Shift^l\pi$ is an $(s+l)$-quasi-isomorphism.
\end{lemm}
\begin{proof}
  The proof of the first statement is similar to the proof of \cref{suppressiveEquivalence}.
  The proof of the second is \cite[Lemma 3.21]{CELW}.
\end{proof}
In particular note that $\Shift^l$ sends $S$-acyclic fibrations to $(S+l)$-fibrations.
For a model category $\mc{M}$ we denote the subcategory of its cofibrant objects by $\mc{M}^{\mathrm{Cof}}$.
We can now prove the following proposition.
\begin{prop}
  There is an equivalence of categories:
  \begin{equation*}
    \adj{\Shift^l}{\fChCof_S}{\fChCof_{S+l}}{\Dec^l}
  \end{equation*}
  and furthermore they are inverse to each other.
\end{prop}
\begin{proof}
  Most of this proposition follows from \cref{suppressiveEquivalence}.
  It only remains to check that $\Dec^l$ sends an $(S+l)$-cofibrant object to an $S$-cofibrant object.
  Consider then a lifting problem of the form:
  \begin{equation*}
    \begin{tikzcd}
      &Y\arrow[d,"\pi",two heads, "\sim"']\\
      \Dec^lA\arrow[r]&X
    \end{tikzcd}
  \end{equation*}
  where $A$ is an $(S+l)$-cofibrant object and $\pi$ an $S$-acyclic fibration.
  Applying $\Shift^l$ to this diagram and using $\Shift^l\circ\Dec^l(B)=B$ since $B$ is $(S+l)$-cofibrant and hence $(s_0+l)$-suppressive so $l$-suppressive we obtain the lifting problem:
  \begin{equation*}
    \begin{tikzcd}
      &\Shift^lY\arrow[d,"\Shift^l\pi",two heads, "\sim"']\\
      A\arrow[r]\arrow[ru,dotted,"h"]&\Shift^lX
    \end{tikzcd}
  \end{equation*}
  which has a lift: the morphism $\Shift^l\pi$ is an $(S+l)$-acyclic fibration by \cref{shiftPreservesAcyclicFibrations} and since $A$ is $(S+l)$-cofibrant there is a lift $h$.
  A lift in the original diagram is then given by $\Dec^lh$.
\end{proof}

%% file: tex/monoidal.tex
\section{Monoidal model structures on $\fChS$}\label{monoidal}
In this section we construct a cofibrant replacement of the monoidal unit, verify it satisfies the (very) strong unit axiom and the pushout-product axioms thereby yielding that the $\fChS$ are monoidal model categories.
\subsection{A cofibrant replacement of the unit}
In the model category structures on multicomplexes of \cite{FGLW} the authors show the unit $R^{0,0}$ is not cofibrant and give a cofibrant replacement, which for bicomplexes is of the form of an `infinite staircase', \cite[Example 6.6 and Proposition 6.7]{FGLW}.
Similarly in $\fChS$ we show the monoidal unit $R_{(0)}^0$ is not cofibrant and as in the case for multicomplexes our cofibrant replacement resembles an infinite staircase.
We first construct a weakly equivalent filtered chain complex and show it is cofibrant in $\fChS$.
Their cofibrant replacement is independent of the value $r$ defining the weak equivalences however ours will depend on it.

\begin{defi}\label{rCofibrantUnit}
  The filtered chain complex $Q_r\I$ is given by
  \begin{equation*}
    Q_r\I \coloneqq \left(\bigoplus_{i=0}^\infty R_{(-i)}^0 \longrightarrow
      \bigoplus_{j=1}^\infty R_{(-r-j)}^1\right)
  \end{equation*}
  where the differential is given for $i\geq 1$ by mapping each $R_{(-i)}^0$ diagonally onto the copies of $R$ indexed by $R_{(-i-r)}^1$ and $R_{(-i-r-1)}^1$ and for $i=0$ by mapping $R_{(0)}^0$ identically onto $R_{(-r-1)}^1$.
\end{defi}
This is more easily pictured as in: 
\begin{equation*}
  \begin{tikzcd}[row sep=0.5em]
    R_{(0)}^0 \arrow[r] & R_{(-r-1)}^1 \\ 
    R_{(-1)}^0 \arrow[r] \arrow[ur] & R_{(-r-2)}^1 \\
    R_{(-2)}^0 \arrow[r] \arrow[ur] & R_{(-r-3)}^1 \\
    {\makesamewidth[c]{$_{(-1)}R_0$}{$\vdots$}} \arrow[r] \arrow[ur] &
    {\makesamewidth[c]{$_{(-r-1)}R_{-1}$}{$\vdots$}}
  \end{tikzcd}
\end{equation*}
where all arrows are identity morphisms and gives a resemblance to the infinite staircase cofibrant replacement of \cite{FGLW}.
We denote by $1_{(p)}^n$ a generator of the $R_{(p)}^n$ component of $Q_r\I$ for the appropriate indices.

Intuition for $Q_r\I$ being a cofibrant replacement of the unit is as follows: $R_{(0)}^0$ is not cofibrant so we replace it with $R_{(0)}^0\rightarrow R_{(-r-1)}^1$ which is cofibrant, however this is no longer weakly equivalent to $R_{(0)}^0$.
We need to kill the element introduced by the $R_{(-r-1)}^1$ term so we form:
\begin{equation*}
  \begin{tikzcd}[row sep=0.5em]
    R_{(0)}^0 \arrow[r] & R_{(-r-1)}^1 \\ 
    R_{(-1)}^0 \arrow[ru] &
  \end{tikzcd}
\end{equation*}
which is again weakly equivalent to $R_{(0)}^0$ but now not cofibrant.
We can correct for this by introducing the $R_{(-r-2)}^1$ term yielding something cofibrant but not weakly equivalent to $R_{(0)}^0$.
Iterating these two procedures ad infinitum will fix both these issues as we will show.

\begin{defi}\label{rCofibrantUnitReplacement}
  We denote by $\pi_r\colon Q_r\I\rightarrow R_{(0)}^0$ the morphism of filtered chain complexes which is the projection of $Q_r\I$ onto its $R_{(0)}^0$ component.
\end{defi}
\begin{exam}\label{QrISS}
  We consider the spectral sequence associated to $Q_r\I$.
  The $0$-page is given by $E_0^{p,p+n}=F_pQ_r\I^n/F_{p-1}Q_r\I^n$ and so there is a copy of $R$ in bidegrees $(p,p)$ for $p\leq 0$ and $(p-r,p-r+1)$ for $p\leq-1$.
  Consider then the $k$-page differential $d_k$ of bidegree $(-k,-k+1)$.
  For $k<r$ these must all be $0$ for degree reasons.
  For $k=r$ there are potentially non-zero differentials between the $R$ in bidegrees $(p,p)$ and $(p-r,p-r+1)$ for $p\leq -1$.
  We show these are all identity differentials.

  The elements $1_{(p)}^0$ are representatives for the copies of $R$ in bidegrees $(p,p)$ for $p\leq-1$ on the $r$-page.
  Its differential is given by $1_{(p-r)}^1+1_{(p-r-1)}^1$ whose image class in $E_r^{p-r,p-r+1}\cong E_0^{p-r,p-r+1}=F_{p-r}Q_r\I^1/F_{p-r-1}Q_r\I^1$ is the same as that of $1_{(p-r)}^1$ hence the $r$-page differentials from $(p,p)$ with $p\leq -1$ are identities.
\end{exam}
\begin{lemm}\label{rCofibrantUnitReplacementAcyclicFibration}
  The morphism $\pi_r$ is an $S$-acyclic fibration.
\end{lemm}
\begin{proof}
  The morphism is an $S$-fibration since the component $R_{(0)}^0$ of $Q_r\I$ is a $k$-cycle for all $k\leq r+1$.
  The preceding example shows that it is in addition an $r$-weak equivalence.
\end{proof}
We now need to show that $Q_r\I$ is cofibrant.
Note first there is a change of basis on the filtered chain complexes $\left(R_{(p)}^n\overset{\Delta}{\rightarrow} R_{(p)}^{n+1}\oplus R_{(p-1)}^{n+1}\right)$ compatible with the filtration.
By this we mean that the isomorphism exhibiting the change of basis and its inverse are morphisms of filtered chain complexes.
The change of basis is depicted as the vertical morphisms in Diagram \ref{eq:ChangeOfBasis}.
\begin{equation}
  \label{eq:ChangeOfBasis}
  \begin{tikzcd}[ampersand replacement=\&]
    \left(R_{(p)}^n\arrow[r,"\Delta"]\arrow[d,"1"',dashed]\right.\&
    \left.R_{(p)}^{n+1}\oplus R_{(p-1)}^{n+1}\right)\arrow[d,"{\begin{pmatrix}1&0\\1&-1\end{pmatrix}}",dashed]\\
    \left(R_{(p)}^n\arrow[r,"{\begin{pmatrix}1\\0\end{pmatrix}}"]\right.\&
    \left.R_{(p)}^{n+1}\oplus R_{(p-1)}^{n+1}\right)
  \end{tikzcd}
\end{equation}
\begin{rema}
  If one writes $M\coloneqq R_{(p_1)}\oplus R_{(p_2)}\oplus\ldots\oplus R_{(p_k)}$ for some filtered $R$-module with $p_1> p_2> \ldots> p_k$ then a change of basis compatible with filtration is precisely a lower triangular invertible matrix.
\end{rema}
We will apply this change of basis in the proof of the following.
\begin{lemm}\label{lemm:QrIIsCofibrant}
  The filtered chain complex $Q_r\I$ is cofibrant in $\fChS$.
\end{lemm}
\begin{proof}
  Consider a lifting problem of $Q_r\I$ against an $S$-acyclic fibration $f\colon E\rightarrow B$ as in
  \begin{equation*}
    \begin{tikzcd}
      &E\arrow[d,"f",two heads, "\sim"']\\
      Q_r\I\arrow[ur, dashed, "h"]\arrow[r]&B
    \end{tikzcd}\;.
  \end{equation*}
  We will begin by lifting the subcomplex $R_{(0)}^0\rightarrow R_{(-r-1)}^1$ of $Q_r\I$ and then inductively lift the subcomplexes $\Delta\colon R_{(-p)}^0\rightarrow R_{(-p-r)}^1\oplus R_{(-p-r-1)}^1$ subject to already having a lift for the $R_{(-p-r)}^1$ component.
  This inductive stage will make use of the change of basis.

  The subcomplex $R_{(0)}^0\rightarrow R_{(-r-1)}^1$ is a copy of $\Z_{r+1}(0,0)$, which is cofibrant by \cite[Lemma 3.2]{CELW} since pushouts of cofibrations are cofibrations.
  Hence since $f$ is an $S$-acyclic fibration there exists a lift of this $(r+1)$-cycle. Consider now the lifting problem:
  \begin{equation*}
    \begin{tikzcd}[row sep=large]
      R_{(-p-r)}^1\arrow[d,"i"']\arrow[r]
      &E\arrow[d,"f",two heads, "\sim"']\\
      \left(R_{(-p)}^0\overset{\Delta}{\rightarrow} R_{(-p-r)}^1\oplus R_{(-p-r-1)}^1\right)
      \arrow[ur, dashed, "h_p"]\arrow[r]&B
    \end{tikzcd}
  \end{equation*}
  where $i$ is the inclusion of the $R_{(-p-r)}^1$ component.
  Applying the filtered change of basis outlined above converts this diagram to an isomorphic one of the form:
  \begin{equation*}
    \begin{tikzcd}[row sep=large]
      R_{(-p-r)}^1\arrow[d,"\Delta"']\arrow[r]
      &E\arrow[d,"f",two heads, "\sim"']\\
      \left(R_{(-p)}^0\overset{i_1}{\rightarrow} R_{(-p-r)}^1\oplus R_{(-p-r-1)}^1\right)
      \arrow[ur, dashed, "h_p"]\arrow[r]&B
    \end{tikzcd}
  \end{equation*}
  for which a lift $h_p$ exists: the morphism $f$ is an $S$-acyclic cofibration and the morphism $\Delta$ is a particular case of $\phi\colon \Z_{r+1}(-p-r,1)\rightarrow\B_{r+1}(-p-r,1)$ in which the images of the $R_{(-p-2r-1)}^2$ components are $0$.
  One then inducts ``down the staircase step by step'' using these change of bases to construct a lift $h$ of $Q_r\I$ from the individual lifts $h_p$ of the subcomplexes of $Q_r\I$.
\end{proof}
Our cofibrant replacement $Q_r\I$ of the unit then does not actually depend on all elements of $S$, only its maximum.
Note too that $Q_r\I$ being cofibrant demonstrates our subclass of cofibrant objects of $\fChr$ of \cref{subclassOfCofibrantObjects} is indeed a strict subclass as it does not satisfy the bounded filtration condition.

We lastly check that the unit $R_{(0)}^0$ is indeed not cofibrant.
\begin{lemm}\label{UnitIsNotCofibrant}
  The unit $R_{(0)}^0$ of the monoidal structure on $\fCh$ is not cofibrant.
\end{lemm}
\begin{proof}
  Consider the lifting problem:
  \begin{equation*}
    \begin{tikzcd}
      & Q_r\I\arrow[d,"\pi_r", two heads, "\sim"']\\
      R_{(0)}^0\arrow[r, "1"]& R_{(0)}^0
    \end{tikzcd}
  \end{equation*}
  in which the morphism $\pi_r$ is that of \cref{rCofibrantUnitReplacement} and so by \cref{rCofibrantUnitReplacementAcyclicFibration} is an $S$-acyclic fibration.
  If a lift were to exist in this diagram then the image of $R_{(0)}^0$ in $Q_r\I$ would have to land in infinitely many summands of the cohomological degree $0$ part $\bigoplus_{p\leq 0} R_{(-p)}^0$ of $Q_r\I$ for its differential to remain $0$.
  This is not possible by definition of the coproduct.
  Hence no lift can exist and so $R_{(0)}^0$ is not cofibrant. 
\end{proof}
\subsection{The unit axiom}
We will demonstrate the unit axiom holds in $\fChS$, in fact we demonstrate the very strong unit axiom holds, recall \cref{veryStrongUnitAxiom}.
We begin by describing cycle and boundary objects in $Q_r\I\otimes A$ for an $A\in\fCh$.
Recall the notation $1_{(p)}^n$ for the generators of the summands $R_{(p)}^n$ of $Q_r\I$ of \cref{rCofibrantUnit}.
An element $q\in F_p\left(Q_r\I\otimes A\right)^n$ can be written as a finite sum of tensors:
\begin{equation}\label{qOfProduct}
  q = \sum_{k\geq 0}1_{(-k)}^0 \otimes a_{(p+k)}^n +
  \sum_{j\geq 0} 1_{(-j-r-1)}^1 \otimes a_{(p+j+r+1)}^{n-1}
\end{equation}
where $a_{(p)}^n\in F_pA^n$.
The differential $d=d^{Q_r\I\otimes A}$ of the monoidal product of $Q_r\I$ and $A$ applied to $q$ is given by:
\begin{equation*}
  dq = \sum_{k\geq 0}1_{(-k)}^0\otimes da_{(p+k)}^n + \sum_{j\geq 0}1_{(-j-r-1)}^1\otimes
  \left(-da_{(p+j+r+1)}^{n-1}+a_{(p+j)}^n+a_{(p+j+1)}^n\right)\,.
\end{equation*}
We now give a description of the $(r+1)$-cycles and boundaries  of $Q_r\I\otimes A$.
The following two lemmas follow from the definitions of cycles, boundaries and the tensor product in $\fCh$.
\begin{lemm}\label{QrIACycles}
  Let $q\in F_p(Q_r\I\otimes A)$ be given as in \cref{qOfProduct}, then $q\in Z_{r+1}^{p,p+n}(Q_r\I\otimes A)$ if and only if the following hold:
  \begin{enumerate}
  \item $a_{(p)}^n\in Z_{r+1}^{p,p+n}(A)$, 
  \item $a_{(p+k)}^n\in B_{r+1}^{p+k,p+k+n}(A)$, for each $k\geq 1$, and
  \item $a_{(p+k)}^n-da_{(p+r+k)}^{n-1}\in F_{p-1+k}A^n$, for each $k\geq1$.\qed
  \end{enumerate}
\end{lemm}

\begin{lemm}\label{QrIABoundaries}
  Let $q\in F_p(Q_r\I\otimes A)$ be given as in \cref{qOfProduct}, then $q\in B_{r+1}^{p,p+n}(Q_r\I\otimes A)$ if and only if $q$ satisfies:
  \begin{enumerate}
  \item $a_{(p)}^n\in B_{r+1}^{p,p+n}(A)$, 
  \end{enumerate}
  in addition to conditions 2 and 3 of \cref{QrIACycles}.\qed
\end{lemm}
Note that in the above lemmas condition $3$ is giving an explicit representation of $a_{(p+k)}^n$ as an $(r+1)$-boundary of $A$ so that in fact $3 \Rightarrow 2$.
\begin{prop}\label{VeryStrongUnitAxiomForFCHS}
  The morphism $\pi_r\colon Q_r\I\overset{\sim}{\twoheadrightarrow} R_{(0)}^0$ of \cref{rCofibrantUnitReplacement} satisfies the very strong unit axiom in $\fChS$.
\end{prop}
\begin{proof}
  We must show the composite morphism $Q_r\I\otimes A\rightarrow R_{(0)}^0\otimes A\rightarrow A$ is a weak equivalence for all (not necessarily cofibrant) $A$, i.e.\ that the map:
  \begin{equation}\label{QrIAMaprPage}
    E_{r+1}^{p,p+n}(Q_r\I\otimes A)= \frac{Z_{r+1}^{p,p+n}(Q_r\I\otimes A)}{B_{r+1}^{p,p+n}(Q_r\I\otimes A)}
    \longrightarrow\frac{Z_{r+1}^{p,p+n}(A)}{B_{r+1}^{p,p+n}(A)}=E_{r+1}^{p,p+n}(A)
  \end{equation}
  is an isomorphism.
  Note that Morphism \ref{QrIAMaprPage} is surjective by surjectivity of $Z_{r+1}^{p,p+n}(Q_r\I\otimes A)\rightarrow Z_{r+1}^{p,p+n}(A)$: for any $(r+1)$-cycle $a$ of $A$ it is the image of $1_{(0)}^0\otimes a$ which is itself an $(r+1)$-cycle of $Q_r\I\otimes A$.

  For injectivity let $q\in Z_{r+1}^{p,p+n}(Q_r\I\otimes A)$ be of the form of \cref{qOfProduct} and such that its image under $Q_r\I\otimes A\rightarrow A$ is a boundary, i.e.\ $a_{(p)}^n\in B_{r+1}^{p,p+n}(A)$ by \cref{QrIABoundaries}.
  We pick an explicit representation of $a_{(p)}^n$ as a boundary, say $da_{(p+r)}^{n-1}+b_{(p-1)}^n$.
  Recall too from \cref{QrIACycles} that condition $3$ gave explicit representatives, for $k\geq 1$, for $a_{(p+k)}^n$ as $(r+1)$-boundaries: $a_{(p+k)}^n=da_{(p+r+k)}^{n-1}+b_{(p-1+k)}^n$.
  For convenience write set $a_{(p+k+r)}^{n-1}=0$.

  For $k\geq 0$ we have the following equality:
  \begin{align}\label{BoundaryElementApkn}
    1_{(-k)}^0\otimes a_{(p+k)}^n &= d\left(1_{(-k)}^0\otimes a_{(p+k+r)}^{n-1}\right) +1_{(-k)}^0\otimes b_{(p+k-1)}^n\nonumber\\
                                  & \phantom{--}-1_{(-k-r)}^1\otimes a_{(p+k+r)}^{n-1}-1_{(-k-r-1)}^1\otimes a_{(p+k+r)}^{n-1}
  \end{align}
  where the penultimate term is interpreted as $0$ if $k=0$.
  Rearranging \cref{BoundaryElementApkn} and summing over $k\geq 0$ gives:
  \begin{align*}
    q &= \sum_{k\geq0}\left(1_{(-k)}^0\otimes a_{(p+k)}^n+1_{(-k-r)}^1\otimes a_{(p+k+r)}^{n-1}\right)\\
      &= \sum_{k\geq0}\left(d\left(1_{(-k)}^0\otimes a_{(p+k+r)}^{n-1}\right)+1_{(-k)}^0\otimes b_{(p+k-1)}^n
        -1_{(-k-r-1)}^1\otimes a_{(p+k+r)}^{n-1}\right)\;.
  \end{align*}
  We now identify terms in this second expression for $q$. The terms of the form $d\left(1_{(-k)}^0\otimes a_{(p+k+r)}^{n-1}\right)$ are $d$ applied to an element of $Z_{r}^{p+r,p+r+n-1}(Q_r\I\otimes A)$ and those of the form $1_{(-k)}^0\otimes b_{(p+k-1)}^n-1_{(p+k-1)}^1\otimes a_{(p+k+r)}^{n-1}$ are elements of $Z_{r}^{p-1,p-1+n}(Q_r\I\otimes A)$. We have then written $q$ as an element of $B_{r+1}^{p,p+n}(Q_r\I\otimes A)$. This shows injectivity of Morphism \ref{QrIAMaprPage}.
\end{proof}
\subsection{The pushout-product axiom}
We turn our attention now to verifying the pushout-product axiom of \cref{monoidalModelCategory}.
Recall from \cref{PushoutProductOnGeneratingCofibrations} that, since the $\fChS$ are symmetric monoidal (as categories) and cofibrantly generated, we need only check that $I_S\boxtimes I_S\subseteq I_S\Cof$, $I_S\boxtimes J_S\subseteq J_S\Cof$ and $J_S\boxtimes J_S\subseteq J_S\Cof$.
The latter two conditions will be easily checked the first however is a somewhat involved computation.
The following lemmas prove the two conditions involving $J_S$.

\begin{lemm}\label{PushoutProductOfJAndJIsAcyclicCofibration}
  Let $s\leq t$. There is an isomorphism of filtered chain complexes:
  \begin{equation*}
    \Z_t(q,m)\otimes\Z_s(p,n)\cong \Z_s(p+q,n+m)\oplus\Z_s(p+q-t,n+m+1)\;\pushQED{\qed}\qedhere
  \end{equation*}
\end{lemm}

There is a similar decomposition for the pushout product of $\Z_{r+1}(p,n)\rightarrow\B_{r+1}(p,n)$ and $0\rightarrow \Z_s(q,m)$ given in the following lemma.

\begin{lemm}\label{PushoutProductOfIAndJIsAcyclicCofibration}
  Let $s\leq r$. There is an isomorphism of morphisms of filtered chain complexes between the pushout-product:
  \begin{equation*}
    \begin{pmatrix}
      \Z_{r+1}(p,n)\\\downarrow\\\B_{r+1}(p,n)
    \end{pmatrix}\boxtimes
    \begin{pmatrix}
      0\\\downarrow\\\Z_s(q,m)
    \end{pmatrix}\cong
    \begin{pmatrix}
      \Z_{r+1}(p,n)\otimes\Z_s(q,m)\\
      \downarrow\\
      \B_{r+1}(p,n)\otimes\Z_s(q,m)
    \end{pmatrix}
  \end{equation*}
  and the direct sum of morphisms (where the latter two are identity morphisms):
  \begin{align*}
    \begin{pmatrix}
      0\\\downarrow\\Z_s(p+q+r,n+m-1)
    \end{pmatrix}&\oplus
                   \begin{pmatrix}
                     0\\\downarrow\\\Z_s(p+q-1,n+m)
                   \end{pmatrix}\\
                 &\oplus
                   \begin{pmatrix}
                     \Z_s(p+q-r-1,n+m+1)\\\downarrow\\\Z_s(p+q-r-1,n+m+1)
                   \end{pmatrix}\\
                 &\oplus
                   \begin{pmatrix}
                     \Z_s(p+q,n+m)\\\downarrow\\\Z_s(p+q,n+m)
                   \end{pmatrix}\;.
  \end{align*}
\end{lemm}
\begin{proof}
  The proof is in \cref{ProofOfPushoutProductOfIAndJIsAcyclicCofibration}.
\end{proof}
These lemmas prove the following.
\begin{coro}\label{PushoutProductWithAJIsCofibration}
  In the model category $\fChS$ morphisms of the pushout-products $I_S\boxtimes J_S$ and $J_S\boxtimes J_S$ are elements of $J_S\Cof$.\qed
\end{coro}
The case $I_S\boxtimes I_S\subseteq I_S$ involves computing pushouts in filtered chains.
Recall there is an adjunction $\inadj{\rho}{\Ch^{\Zbb_\infty}}{\fCh}{i}$ and by \cref{filteredChainsLimits} we can instead compute colimits in the functor category $\Ch^{\Zbb_\infty}$.
Given a diagram $X\colon \mc{J}\rightarrow\fCh$ its colimit is given by $\rho\colim_{\mc{J}}iX$.
\begin{prop}\label{pushoutProductOfGeneratingCofibrations}
  The pushout-product of $i\colon\Z_{r+1}(p,n)\rightarrow\B_{r+1}(p,n)$ and $j\colon\Z_{r+1}(q,m)\rightarrow\B_{r+1}(q,m)$ is given as in \cref{pushoutProduct} with named generators given in \cref{ABCPushoutDiagram}, differentials on generators described in \cref{namedDifferentialsOfPushoutroduct} and the map $i\boxtimes j$ described on generators given in \cref{mapOnGeneratorInducedBy}.
\end{prop}
\begin{proof}
  The proof is in \cref{ProofOfpushoutProductOfGeneratingCofibrations}.
\end{proof}
\begin{coro}\label{PushoutProductOfIAndIIsCofibration}
  \sloppy The pushout-product of generating cofibrations $i\colon\Z_{r+1}(p,n)\rightarrow\B_{r+1}(p,n)$ and $j\colon\Z_{r+1}(q,m)\rightarrow\B_{r+1}(q,m)$ is a cofibration in $\fChS$.
\end{coro}
\begin{proof}
  We demonstrate the pushout-product is a cofibration by appealing to our partial classification \cref{subclassOfCofibrations}.
  We must then demonstrate that $i\boxtimes j$ is an $r$-suppressive inclusion with cofibrant cokernel which also satisfies condition 5 of \cref{subclassOfCofibrantObjects}.

  Observe that the morphism $i\boxtimes j$ is equivalently a morphism of the form $X\rightarrow X\toplus Y$ where $Y$ is given, up to signs, by the filtered chain complex:
  \begin{equation*}
    \begin{tikzcd}
      R_{(p+r)+(q+r)}^{(n-1)+(m-1)}\arrow[r,"-1"]\arrow[d]&
      R_{(p+r)+(q-1)}^{(n-1)+(m)}\arrow[d]\\
      R_{(p-1)+(q+r)}^{(n)+(m-1)}\arrow[r]&R_{(p-1)+(q-1)}^{(n)+(m)}
    \end{tikzcd}
  \end{equation*}
  which decomposes into the direct sum of two $r$-cycles by \cref{PushoutProductOfJAndJIsAcyclicCofibration}.
  This is then a cofibrant object in $\fChr$ with bounded below filtration.
  Further by studying \cref{pushoutProduct} one sees that the twist differential $\tau$ suppresses filtration by $r$.

  Together these show the conditions needed to apply \cref{subclassOfCofibrations} and so the pushout-product $i\boxtimes j$ is a cofibration in $\fChr$.
  Finally note that since $I_r\subseteq I_S$ we have $I_r\Cof\subseteq I_S\Cof$.
\end{proof}
\subsection{Monoidal model categories}
\begin{theo}\label{fChSIsMonoidal}
  The model categories $\fChS$ are closed symmetric monoidal model categories.
\end{theo}
\begin{proof}
  We must show the unit and pushout-product axioms of \cref{monoidalModelCategory}.
  The unit axiom was proved in \cref{VeryStrongUnitAxiomForFCHS}.
  The pushout-product axiom was proved for $I_S\boxtimes I_S\subseteq I_S\Cof$ in \cref{PushoutProductOfIAndIIsCofibration} and for $I_S\boxtimes J_S\subseteq J_S\Cof$  and $J_S\boxtimes J_S\subseteq J_S\Cof$ in \cref{PushoutProductWithAJIsCofibration}.
\end{proof}
Write $Q\colon \fChS\rightarrow \fChS$ for a cofibrant replacement functor.
\begin{coro}\label{homotopyCategoryfChSIsMonoidal}
  The homotopy category $\Ho(\fChS)$ has an induced closed symmetric monoidal structure with monoidal product given by:
  \begin{equation*}
    A\Lotimes B\coloneqq QA\otimes QB\,,
  \end{equation*}
  and internal hom object functor:
  \begin{equation*}
    R\Hom(A,B) \coloneqq \Hom(QA,B)\,.
  \end{equation*}
  We need not fibrantly replace $B$ in the hom objects since all objects of $\fChS$ are fibrant. The unit in the homotopy category is given by any cofibrant replacement of $R_{(0)}^0$, in particular we can take $Q_r\I$ from \cref{rCofibrantUnit}.
\end{coro}
\begin{proof}
  This follows from \cite[Theorem 4.3.2]{H} applied to the closed symmetric monoidal model category $\fChS$ of \cref{fChSIsMonoidal}.
\end{proof}

%% file: tex/algebras.tex
\section{Model categories of algebras and modules}\label{algebras}
The last section established closed symmetric monoidal model category structures on filtered chain complexes given by $\fChS$.
We now construct model structures on algebras and modules in $\fCh$ from the $\fChS$.
For this we demonstrate Schwede and Shipley's monoid axiom given in \cref{MonoidAxiom}.
Recall by \cref{MonoidAxiomFromJ} that it suffices to check the monoid axiom for cellular morphisms formed from $J_S\otimes \fCh$.
\begin{lemm}\label{JSOtimesFCHIsWeakEquivalence}
  Any morphism of $J_S\otimes\fCh$ is an $r$-weak equivalence.
\end{lemm}
\begin{proof}
  Any morphism of $J_S\otimes\fCh$ is of the form $\left(0\rightarrow\Z_s(p,n)\right)\otimes A$ for some $A\in\fCh$.
  Up to a shift of cohomological and filtration degree this is of the form $0\rightarrow C_s(A)$ where $C_s(-)$ is the $s$-cone of \cref{rCone}. By \cref{rConeIsrAcyclic} the $s$-cone is $s$-acyclic, and hence $r$-acyclic.
\end{proof}
\begin{lemm}
  Any pushout of a filtered chain complex $A$ by a morphism of $J_S\otimes\fCh$ is an $r$-weak equivalence.
\end{lemm}
\begin{proof}
  Such a morphism is of the form $A\rightarrow A\oplus\left(\Z_s(p,n)\otimes B\right)$ which is a weak equivalence by \cref{JSOtimesFCHIsWeakEquivalence}.
\end{proof}
\begin{prop}
  Every morphism of $\left(J_S\otimes\fCh\right)\Cell$ is an $r$-weak equivalence.
\end{prop}
\begin{proof}
  Such a morphism is of the form $A\rightarrow A\oplus\bigoplus_i C_{s(i)}(B_i)$ where $B_i\in\fCh$ and $s(i)\in S$.
  By \cref{rConeIsrAcyclic} the cones are $s(i)$-acyclic and hence also $r$-acyclic so that the morphism is an $r$-quasi-isomorphism.
\end{proof}
\begin{coro}\label{MonoidAxiomInfChS}
  The model categories $\fChS$ satisfy the monoid axiom of \cref{MonoidAxiom} of Schwede and Shipley.\qed
\end{coro}
Given the model categories $\fChS$ are closed symmetric monoidal model categories which are cofibrantly generated and satisfy the monoid axiom the following theorems are all immediate consequences of \cref{SSModelStructures} of Schwede and Shipley since all objects of $\fCh$ are small relative to $\fCh$, \cref{filtChainsAreSmall}.

Recall the definitions of filtered differential graded (commutative) algebras, their (left) modules and categories from \cref{fdgas,filteredDGModules,categoryOfFDGAS}.
The set $S$ will continue to denote a subset of $\{0,1,2,\ldots,r\}$ containing $r$.
Denote by $U$ the forgetful functor from one of the categories of modules or algebras below to the category $\fCh$.
\begin{theo}\label{leftAModulesModCat}
  For $A\in\fdga$ there is a cofibrantly generated model category structure on the category of left $A$-modules whose weak equivalences are those morphisms that are $r$-quasi-isomorphisms after applying $U$ and fibrations are those morphisms that are surjective on $s$-cycles for all $s\in S$ after applying $U$.

  The generating cofibrations and acyclic cofibrations are given by $A\otimes I_S$ and $A\otimes J_S$ respectively.\qed
\end{theo}
\begin{theo}\label{AModulesModCat}
  For $A\in\fdgc$ there is a cofibrantly generated model category structure on the category of $A$-modules whose weak equivalences are those morphisms that are $r$-quasi-isomorphisms after applying $U$ and fibrations are those morphisms that are surjective on $s$-cycles for all $s\in S$ after applying $U$.

  The generating cofibrations and acyclic cofibrations are given by $A\otimes I_S$ and $A\otimes J_S$ respectively.\qed
\end{theo}
For $A\in\fdgc$ denote by $T_A$ the free $A$-algebra functor $T_A\colon\fCh\rightarrow\fdga$.
\begin{theo}\label{AAlgebrasModCat}
    For $A\in\fdgc$ there is a cofibrantly generated model category structure on the category of $A$-algebras whose weak equivalences are those morphisms that are $r$-quasi-isomorphisms after applying $U$ and fibrations are those morphisms that are surjective on $s$-cycles for all $s\in S$ after applying $U$.

  The generating cofibrations and acyclic cofibrations are given by $T_AI_S$ and $T_AJ_S$ respectively.\qed
\end{theo}
In particular setting $A$ to be the filtered differential graded commutative algebra $R_{(0)}^0$ in \cref{AAlgebrasModCat} we obtain a model category of filtered differential graded algebras.
Denote by $T$ the free algebra functor $T_{R_{(0)}^0}$.
\begin{coro}\label{fdgaModCat}
  The category $\fdga$ admits a model category structure whose weak equivalences are those morphisms that are $r$-quasi-isomorphisms after applying $U$ and fibrations those morphisms that are surjective on $s$-cycles after applying $U$.

  The generating cofibrations and acyclic cofibrations are given by $TI_S$ and $TJ_S$ respectively.\qed
\end{coro}


%% file: tex/cofibrant-unit.tex
\section{Monoidal model structures on $\fCh$ with cofibrant unit}\label{cofibrant-unit}
Recall from \cref{UnitIsNotCofibrant} that the unit $R_{(0)}^0$ of the monoidal model category $\fChS$ is not cofibrant.
A cofibrant unit is often desirable so we construct here, using the results of Muro outlined previously in \cref{MurosResults}, further new model structures on filtered chain complexes for which the unit becomes cofibrant at the cost of making the generating (acyclic) cofibrations more difficult to contend with.

Recall the morphism $\pi_r\colon Q_r\I\rightarrow R_{(0)}^0$ of \cref{rCofibrantUnitReplacement} which is an $S$-acyclic fibration by \cref{rCofibrantUnitReplacementAcyclicFibration} with $Q_r\I$ a cofibrant replacement for the unit in the $S$-model structure.
In Muro's setup we must find a factorisation:
\begin{equation*}
  Q_r\I\oplus R_{(0)}^0 \rightarrowtail C \overset{\sim}{\rightarrow} R_{(0)}^0
\end{equation*}
of $(\pi_r,id)$ into a cofibration followed by a weak equivalence.

\begin{defi}
  We denote by $\rho_r$ the degree $1$ morphism of filtered chain complexes defined as the composite of the $\rSusp\pi_r$ and the morphism $id_{-r}^1$:
  \begin{equation*}
    \rho_r\coloneqq id_{-r}^1\circ \rSusp\pi_r
        \colon \rSusp Q_r\I \overset{\sim}{\rightarrow}
        \rSusp R_{(0)}^0 \simeq R_{(r)}^{-1}
        \rightarrow R_{(0)}^0
      \end{equation*}
      where the latter is the identity on the underlying $R$-modules between degrees $-1$ and $0$ decreasing filtration by $r$.
\end{defi}

\begin{defi}
  We define the filtered chain complex $D$ by:
  \begin{equation*}
    D\coloneqq \left(Q_r\I\oplus R_{(0)}^0\right)\toplus
    \rSusp Q_r\I
  \end{equation*}
  where the twist differential $\tau$ maps $\rSusp Q_r\I$ identically onto the $Q_r\I$ summand (decreasing filtration by $r$ in the process) and onto $R_{(0)}^0$ by $\rho_r$.
\end{defi}

We can then give our factorisation of $Q_r\I\oplus R_{(0)}^0\rightarrow R_{(0)}^0$ into the composite of an $r$-cofibration and $r$-weak equivalence as
\begin{equation*}
  \begin{tikzcd}
    Q_r\I\oplus R_{(0)}^0\arrow[r,"j"]
    & D \arrow[r,"q"]
    &R_{(0)}^0
  \end{tikzcd}
\end{equation*}
where $j$ is the inclusion of $Q_r\I\oplus R_{(0)}^0$ into the first component of the twisted direct sum of $D$.
The morphism $q$ is given by $\nabla\circ (\pi_r\oplus id)$ on the first component of the twisted direct sum and by $0$ on the second.

\begin{lemm}
  The morphism $j\colon Q_r\I\oplus R_{(0)}^0\rightarrow D$ is an $r$-cofibration.
\end{lemm}
\begin{proof}
  The cokernel of $j$ is $\rSusp Q_r\I$ which is a cofibrant object in $\fChr$.
  The twisting differential $\tau$ is $r$-suppressive by construction and so our conditions of \cref{subclassOfCofibrations} hold with the exception of that of a bounded filtration.
  Note however that the object $\rSusp Q_r\I$ can be written (degree wise) as a direct sum, as how it was defined in \cref{rCofibrantUnit}, and that the proof of \cref{subclassOfCofibrations} follows through in the same way with this decomposition.
\end{proof}

\begin{lemm}
  The morphism $q$ is an $r$-weak equivalence.
\end{lemm}
\begin{proof}
  The filtered chain complex $D$ is the cone of the morphism $(id,\pi_r)\colon Q_r\I \rightarrow Q_r\I\oplus R_{(0)}^0$.
  Hence using \cite[Remark 3.6]{CELW}:
  \begin{equation*}
    E_r^{p,q}(D) \cong E_r^{p-r,q+1-r}(Q_r\I)\oplus E_r^{p,q}(Q_r\I)
    \oplus E_r^{p,q}(R_{(0)}^0)
  \end{equation*}
  and the map $E_r^{p,q}(q)\colon E_r^{p,q}(D)\rightarrow E_r^{p,q}(R_{(0)}^0)$ is $0$ except on $E_r^{0,0}(D)\cong E_r^{p-r,q+1-r}(Q_r\I)\oplus R\oplus R$ where it is the identity map from the second and third components to $E_r^{0,0}(R_{(0)}^0)\cong R$.
  The $r$-page differential on $E_r(D)$ can be checked to kill the $E_r^{\ast,\ast}(Q_r\I)$ terms leaving only the $E_r^{\ast,\ast}(R_{(0)}^0)$ term yielding an isomorphism on the $(r+1)$-page.
\end{proof}

We can then give our cofibrantly generated monoidal model structures on $\fCh$ which have a cofibrant unit.

\begin{theo}
  For every $r\geq 0$ and every subset $S\subseteq \{0,1,\ldots,r\}$ containing $r$, the category $\fCh$ admits a right proper cofibrantly generated monoidal model structure denoted $\fChHatS$ satisfying the monoid axiom with:
  \begin{itemize}
  \item weak equivalences given by the $r$-quasi-isomorphisms,
  \item generating cofibrations and acyclic cofibrations given by:
    \begin{align*}
      \tilde{I}_S&\coloneqq I_S\cup\left\{0\rightarrow R_{(0)}^0\right\}\,,\\
      \tilde{J}_S&\coloneqq J_S\cup \left\{j\circ i_1\colon Q_r\I
                   \rightarrow D\right\}
    \end{align*}
    respectively.
  \end{itemize}
\end{theo}
\begin{proof}
  This follows from Muro's theorem, \cref{MurosTheorem}, noting that the conditions are satisfied in our setup: $\mc{M}=\fChS$ is a cofibrantly generated monoidal model category with cofibrant replacement of the unit given by $q\coloneqq \pi_r\colon Q_r\I\rightarrow R_{(0)}^0$ satisfying the very strong unit axiom, \cref{VeryStrongUnitAxiomForFCHS}.
\end{proof}

\begin{rema}
  Note this only makes $R_{(0)}^0$ cofibrant but the $R_{(p)}^n$ in general are not made cofibrant in this procedure.
  One can add in the relevant shifts of the morphisms $0\rightarrow R_{(0)}^0$ and $Q_r\I\rightarrow D$ into $\tilde{I}_S$ and $\tilde{J}_S$ respectively (i.e.\ tensor these by $R_{(p)}^n$) to make all $R_{(p)}^n$ cofibrant too.
  One need only check the proof of Muro's theorem still applies in this case.
\end{rema}


%% file: tex/monoidal-proof.tex
\section{Proof of the pushout-product axiom}\label{monoidal-proof}
This appendix contains the proof of \cref{PushoutProductOfIAndJIsAcyclicCofibration}, that decomposes elements of $I_{r}\boxtimes J_s$, and of \cref{pushoutProductOfGeneratingCofibrations}, that computes the pushout-product of two generating cofibrations in $I_{r}$ of $\fChS$ and shows it is a cofibration.

\subsection{Proof of \cref{PushoutProductOfIAndJIsAcyclicCofibration}}\label{ProofOfPushoutProductOfIAndJIsAcyclicCofibration}

We give a change of basis that exhibits the isomorphism of morphisms of filtered chain complexes give in \cref{PushoutProductOfIAndJIsAcyclicCofibration}.
The morphism $\Z_{r+1}(p,n)\otimes\Z_s(q,m)\rightarrow\B_{r+1}(p,n)\otimes\Z_s(q,m)$ can be depicted as in \cref{pushoutProductOfGenCofibAndAcyclicGenCofib} in which names have been given to the generators of each component $R_{(\bullet)}^\bullet$.

One can check that this morphism of filtered chain complexes decomposes as:
\begin{equation*}
  \begin{tikzcd}[column sep=1.8em]
    0\arrow[d,dashed]\arrow[r]&0\arrow[d,dashed]
    &0\arrow[d,dashed]\arrow[r]&0\arrow[d,dashed]\\
    R\left\{h\right\}
    \arrow[r]& R\left\{f+(-1)^{n-1}g\right\}
    &R\left\{e\right\}\arrow[r]
    &R\left\{b+(-1)^nc\right\}\\
    R\left\{B\right\}\arrow[r]\arrow[d,dashed]
    &R\left\{(-1)^{n+1}A\right\}\arrow[d,dashed]
    &R\left\{D\right\}\arrow[r]\arrow[d,dashed]
    &R\left\{B+(-1)^nC\right\}\arrow[d,dashed]\\
    R\left\{b\right\}\arrow[r]
    &R\left\{(-1)^{n+1}a\right\}
    &R\left\{e+f\right\}\arrow[r]
    &R\left\{b+(-1)^nc+(-1)^nd\right\}
  \end{tikzcd}
\end{equation*}
where we have omitted any cohomological or filtration indexing for brevity.
These are the four morphisms give in the lemma.
One must also check that any of the elements $a,b,c,d,e,f,g,h,A,\allowbreak B,C$ and $D$, in cohomological degree $n_1$ and filtration degree $p_1$, can be written as a linear combination of the elements in the diagram whose cohomological and filtration degrees do not exceed $n_1$ and $p_1$ respectively, so that this change of basis is compatible with the filtration.

This completes the proof.

\input{./tex/diagrams/PushoutProductIAndJ.tex}

\subsection{Proof of \cref{pushoutProductOfGeneratingCofibrations}}
\label{ProofOfpushoutProductOfGeneratingCofibrations}

To compute the pushout-product of two generating cofibrations of the form $i\colon\Z_{r+1}(p,n)\rightarrow \B_{r+1}(p,n)$ and $j \colon \Z_{r+1}(q,m) \rightarrow \B_{r+1}(q,m)$ we will make use of the method outlined in the discussion preceding \cref{filteredChainsLimits}: recall we have an adjunction $\inadj{\rho}{\Ch^{\Zbb_\infty}}{\fCh}{i}$ and that a colimit in $\fCh$ can instead be computed in $\Ch^{\Zbb_\infty}$ via this adjunction.
Once we have identified the pushout-product in this way we realise $i\boxtimes j$ as a morphism of the form of \cref{subclassOfCofibrations} which shows it is an $r$-cofibration.

We begin by computing the domain of the pushout-product:
\begin{equation}\label{PushoutOfDomain}
  \begin{tikzcd}
    \Z_{r+1}(p,n)\otimes\Z_{r+1}(q,m)\arrow[r,"k_2"]\arrow[d,"k_1"']
    \arrow[dr,phantom,very near end, "\ulcorner"]
    &\Z_{r+1}(p,n)\otimes\B_{r+1}(q,m)\arrow[d,dashed]\\
    \B_{r+1}(p,n)\otimes\Z_{r+1}(q,m)\arrow[r,dashed]
    &\dom(i\boxtimes j)
  \end{tikzcd}\,.
\end{equation}

The filtered chain complexes of the lower left, upper left and upper right are depicted in \cref{domainPushoutProductComponents} where $i_k$ and $\pi_k$ are inclusions and projections into or off of the $k$-component.
We have used the convention of \cref{tensorProductBracketingNotation} grouping subscript and superscript terms together to help keep track of the components from the tensor products.

We make a general notational remark about how to interpret \cref{domainPushoutProductComponents,namedDomainPushoutProductComponents,namedCodomainPushoutProduct,codomainPushoutProduct,pushoutProduct}.
\begin{nota}
  Since we are working with tensor products of filtered chain complexes we are displaying these diagrams so that the constituent $R$-modules summands in a fixed cohomological degrees can be read off along a diagonal and similarly for filtration degrees along the horizontal direction.
  We label differentials which for the most part are either a sign times the identity, inclusion of a summand or projection from a summand, e.g.\ in \cref{namedDomainPushoutProductComponents} the differential of $M$ is given by $d(M)=(-1)^nN+O$.

  Unlabelled arrows in these diagrams denote an identity matrix and should be interpreted via adjacency, e.g.\ in \cref{namedCodomainPushoutProduct} there is an arrow from the sum of $R\left\{\ul{C}\right\}$ and $R\left\{\ul{E}\right\}$ to the sum of $R\left\{\ul{F}\right\}$,$R\left\{\ul{I}\right\}$,$R\left\{\ul{H}\right\}$ and $R\left\{\ul{K}\right\}$.
  Along the arrow $R\left\{\ul{C}\right\}$ is adjacent to $R\left\{\ul{F}\right\}$ and $R\left\{\ul{E}\right\}$ is adjacent to $R\left\{\ul{I}\right\}$ so this arrow encodes a differential sending $\ul{C}\mapsto\ul{F}$ and $\ul{E}\mapsto\ul{I}$.

  Other more complicated names are given explicit names and described within this appendix.
\end{nota}

The boxes demarcating a copy of a filtered graded $R$-module in  \cref{domainPushoutProductComponents} illustrate the maps of the pushout Diagram \ref{PushoutOfDomain}: a filtered graded $R$-module of the component corresponding to $\Z_{r+1}(p,n)\otimes\Z_{r+1}(q,m)$ demarcated by a dashed box is mapped identically or diagonally into the boxes of the same dashed type in the other two diagrams.

We provide names for the generators of these filtered graded $R$-modules in \cref{namedDomainPushoutProductComponents} where the filtration and cohomological indexing has also been dropped.
In this notation the map $k_1\colon\Z_{r+1}(p,n)\otimes\Z_{r+1}(q,m)\rightarrow\B_{r+1}(p,n)\otimes\Z_{r+1}(q,m)$ of Diagram \ref{PushoutOfDomain} is determined on generators by:
\begin{align}
  k_1(I)=B+C
  && k_1(J)=E
  && k_1(K)=F+G
  && k_1(L)=H\,,\label{PushoutMaps1}
\end{align}
and the map $k_2\colon\Z_{r+1}(p,n)\otimes\Z_{r+1}(q,m)\rightarrow\Z_{r+1}(p,n)\otimes\B_{r+1}(q,m)$ of \cref{PushoutOfDomain} is determined on generators by:
\begin{align}
  k_2(I)= N+ P
  && k_2(J)=Q+T
  && k_2(K)=S
  && k_2(L)=U\,.\label{PushoutMaps2}
\end{align}
One can write down similar equations for the differentials of these generators, e.g.\ $d(M)=O+(-1)^nN$ that can be read off from \cref{namedDomainPushoutProductComponents}.

We now need to compute the pushout $\dom(i\boxtimes j)$.
For this we must use the description of colimits given by \cref{filteredChainsLimits}, i.e.\ to compute a colimit of a diagram $X\colon \mc{J}\rightarrow \fCh$ we view $X$ as an object of $\Ch^{\Zbb_\infty}$ via an inclusion functor, compute the colimit in this category instead and then apply the functor $\rho$ to yield an object of $\fCh$.

Diagrams for the filtered chain complexes $\Z_{r+1}(p,n)\otimes\Z_{r+1}(q,m)$, $\Z_{r+1}(p,n)\otimes\B_{r+1}(q,m)$ and $\B_{r+1}(p,n)\otimes\Z_{r+1}(q,m)$ realised as $\Zbb_\infty$-indexed diagrams are depicted in \cref{ZCrossZInZ+Chains,ZCrossBInZ+Chains,BCrossZInZ+Chains} respectively.
The colimit in $\Zbb_\infty$-indexed chains can then be computed pointwise from these three diagrams along with the data of the maps of \cref{PushoutMaps1,PushoutMaps2}.
The result is shown in \cref{PushoutInZ+Chains} where some identifications have taken place, so that for example we now write $\ol{N}$ for the class of the image of $N$ in the pushout.
The equation $\ol{N}+\ol{P}=\ol{B}+\ol{C}$ is in an example of an identification that takes place in the pushout.
The pushout can now be identified in $\fCh$ by applying the functor $\rho$ to \cref{PushoutInZ+Chains}.
This yields the filtered chain complex given by the upper half of \cref{pushoutProduct} where the the filtered graded $R$-modules correspond to those of \cref{ABCPushoutDiagram}.

The notation $R_{(\ast)+(\ast)}^{(\ast)+(\ast)}$ of \cref{tensorProductBracketingNotation} is somewhat meaningless after these identifications have taken place however we maintain it for clarity of the diagrams.

The differentials labelled in \cref{ABCPushoutDiagram} can also be read off from \cref{PushoutInZ+Chains}.
Those labelled $\alpha,\beta,\gamma$ and $\epsilon$ are given by:
\begin{equation}\label{namedDifferentialsOfPushoutroduct}
  \begin{aligned}
    \alpha(\ol{A})&=\ol{B} & \beta(\ol{M})&=(-1)^n\ol{B}+(-1)^n\ol{C}+(-1)^{n+1}\ol{P}\\
    \gamma(\ol{B})&= 0 & \gamma(\ol{C})&= \ol{Q}+\ol{T}\\
    \gamma(\ol{P})&= \ol{T} & \epsilon(\ol{B})&= (-1)^n\ol{F}\\
    \epsilon(\ol{C})&= (-1)^n\ol{G} & \epsilon(\ol{P})&= (-1)^n\ol{F}+(-1)^n\ol{G}
  \end{aligned}
\end{equation}

We next turn our attention to the map from this pushout into $\B_{r+1}(p,n)\otimes\B_{r+1}(q,m)$.
This codomain $\B_{r+1}(p,n)\otimes\B_{r+1}(q,m)$ is given in \cref{codomainPushoutProduct} and as before we give names to the filtered graded $R$-modules in \cref{namedCodomainPushoutProduct}.
The map $i\boxtimes j$ is induced from the map on the generators of \cref{namedDomainPushoutProductComponents} with values in the generators of 
\cref{namedCodomainPushoutProduct} given by:
\begin{equation}\label{mapOnGeneratorInducedBy}
  \begin{aligned}
    A&\mapsto \ul{C}+\ul{E}
    & B&\mapsto \ul{F}+\ul{I}
    & C&\mapsto \ul{H}+\ul{K}
    & D&\mapsto \ul{J}\\
    E&\mapsto \ul{L}+\ul{N}
    & F&\mapsto \ul{M}
    & G&\mapsto \ul{O}
    & H&\mapsto \ul{P}\\
    M&\mapsto \ul{B}+\ul{D}
    & N&\mapsto \ul{F}+\ul{H}
    & O&\mapsto \ul{G}
    & P&\mapsto \ul{I}+\ul{K}\\
    Q&\mapsto \ul{L}
    & S&\mapsto \ul{M}+\ul{O}
    & T&\mapsto \ul{N}
    & U&\mapsto \ul{P}\\
  \end{aligned}
\end{equation}
which can be checked to be compatible with the maps from the generators $I,J,K$ and $K$ into either component of the pushout.
These maps then induce maps from the pushout in $\Zbb_\infty$-indexed chains to $\B_{r+1}(p,n)\otimes\B_{r+1}(q,m)$ viewed as a $\Zbb_\infty$-indexed chain complex.

This completes the description of the pushout-product $i\boxtimes j$ and the proof of \cref{pushoutProductOfGeneratingCofibrations}.
\begin{equation}\label{greek_i}
  \begin{aligned}
    \alpha_1&=
              \begin{pmatrix}
                1\\(-1)^n
              \end{pmatrix}
    &
      \alpha_2&=
                \begin{pmatrix}
                  (-1)^{n+1}&1
                \end{pmatrix}\\
    \beta_1&=
             \begin{pmatrix}
               (-1)^n\\1\\0
             \end{pmatrix}
    &
      \beta_2&=
               \begin{pmatrix}
                 1&(-1)^{n+1}&0\\
                 0&0&(1)^n\\
                 0&0&1
               \end{pmatrix}\\
    \beta_3&=
             \begin{pmatrix}
               (-1)^{n+1}&0\\
               0&(-1)^n\\
               0&1
             \end{pmatrix}
    &
      \beta_4&=
               \begin{pmatrix}
                 0&1&(-1)^{n+1}
               \end{pmatrix}\\
    \beta_5&=(-1)^{n+1}
    &
      \gamma_1&=
                \begin{pmatrix}
                  1\\0\\(-1)^{n-1}
                \end{pmatrix}\\
    \gamma_2&=
              \begin{pmatrix}
                0&1&0\\
                (-1)^n&0&1\\
                0&(-1)^n&0
              \end{pmatrix}
    &
      \gamma_3&=
                \begin{pmatrix}
                  1&0\\
                  0&1\\
                  (-1)^n&0
                \end{pmatrix}\\
    \gamma_4&=
              \begin{pmatrix}
                (-1)^{n+1}&0&1
              \end{pmatrix}
                            &
                              \gamma_5&=1
  \end{aligned}
\end{equation}
\input{./tex/diagrams/componentsOfPushout.tex}
\input{./tex/diagrams/namedComponentsOfPushout.tex}
\input{./tex/diagrams/namedCodomainOfPushoutProduct.tex}
\input{./tex/diagrams/ZChainsRepForZZ.tex}
\input{./tex/diagrams/ZChainsRepForZB.tex}
\input{./tex/diagrams/ZChainsRepForBZ.tex}
\input{./tex/diagrams/ZChainsRepForPushout.tex}
\input{./tex/diagrams/codomainOfPushoutProduct.tex}
\input{./tex/diagrams/pushoutProduct.tex}
\input{./tex/diagrams/namedPushoutProduct.tex}


%% file: tex/diagrams/PushoutProductIAndJ.tex
\afterpage{%
  \begin{landscape}\centering
    \vspace*{\fill}
    \begin{figure}[htpb]
      \begin{equation*}
        \begin{tikzcd}[ampersand replacement=\&, column sep=0.5cm]
          \&\&\&{R}\vphantom{R}_{(p)+(q)}^{(n)+(m)}\{D\}\arrow[rr,"1"]\arrow[dd,dashed,"\begin{pmatrix}1\\1\end{pmatrix}", near start]
          \arrow[ld,"(-1)^n"]\&\&
          {R}\vphantom{R}_{(p-r-1)+(q)}^{(n+1)+(m)}\{B\}\arrow[dd,dashed,"1"]
          \arrow[dl,"(-1)^{n+1}"]\\
          \&\&{R}\vphantom{R}_{(p)+(q-s)}^{(n)+(m+1)}\{C\}
          \arrow[rr, near start, crossing over,"1"']
          \&\&{R}\vphantom{R}_{(p-r-1)+(q-s)}^{(n+1)+(m+1)}\{A\}\&\\
          \&{R}\vphantom{R}_{(p+r)+(q)}^{(n-1)+(m)}\{h\}
          \arrow[rr, "i_1", near start]
          \arrow[dl,"(-1)^{n-1}"]
          \&\&{R}\vphantom{R}_{(p)+(q)}^{(n)+(m)}\{f\}\oplus {R}\vphantom{R}_{(p-1)+(q)}^{(n)+(m)}\{e\}
          \arrow[dl,"(-1)^n\id"]
          \arrow[rr,"\pi_2", near start]\&\&
          {R}\vphantom{R}_{(p-r-1)+(q)}^{(n+1)+(m)}\{b\}\arrow[dl,"(-1)^{n+1}"]\\
          {R}\vphantom{R}_{(p+r)+(q-s)}^{(n-1)+(m+1)}\{g\}\arrow[rr,"i_1"']
          \&\&{R}\vphantom{R}_{(p)+(q-s)}^{(n)+(m+1)}\{d\}\oplus{R}\vphantom{R}_{(p-1)+(q+s)}^{(n)+(m+1)}\{c\}\arrow[rr,"\pi_2"']
          \arrow[uu, dashed, crossing over, leftarrow,"\begin{pmatrix}1\\1\end{pmatrix}", near end]
          \&\&{R}\vphantom{R}_{(p-r-1)+(q-s)}^{(n+1)+(m+1)}\{a\}\arrow[uu, leftarrow, dashed, crossing over,"1",near end]
        \end{tikzcd}
      \end{equation*}
      \caption{Pushout-product of $\Z_{r+1}(p,n)\rightarrow\B_{r+1}(p,n)$ and $0\rightarrow\Z_s(q,m)$}
      \label{pushoutProductOfGenCofibAndAcyclicGenCofib}
    \end{figure}
    \vfill
  \end{landscape}%
}


%% file: tex/diagrams/componentsOfPushout.tex
\begin{figure}
  \centering
  \begin{equation*}
    \begin{tikzcd}[column sep=2.4cm,
      every matrix/.append style={name=mat},
      execute at end picture={
        \draw[dashed] (mat-5-2.south west) |- (mat-4-2.north east)
        -| (mat-5-2.south east) -- (mat-5-2.south west);
        \draw[dashed] (mat-11-2.north west) -| (mat-11-2.south east)
        -| (mat-11-2.north west);
        \draw[dashed] (mat-16-2.north west) -| (mat-16-2.south east)
        -| (mat-16-2.north west);
        \draw[dotted] (mat-4-3.north west) -| (mat-5-3.south east)
        -- (mat-5-3.south west) |- (mat-4-3.north west);
        \draw[dotted] (mat-11-3.north west) -| (mat-11-3.south east)
        -| (mat-11-3.north west);
        \draw[dotted] (mat-16-3.north west) -| (mat-16-3.south east)
        -| (mat-16-3.north west);
        \draw[dash dot dot] (mat-8-2.north west) -| (mat-8-2.south east)
        -| (mat-8-2.north west);
        \draw[dash dot dot] (mat-13-2.north west) -| (mat-13-2.south east)
        -| (mat-13-2.north west);
        \draw[dash dot dot] (mat-18-2.north west) -| (mat-18-2.south east)
        -- (mat-18-2.south west) |- (mat-18-2.north west);
        \draw[] (mat-8-3.north west) -| (mat-8-3.south east)
        -| (mat-8-3.north west);
        \draw[] (mat-13-3.north west) -| (mat-13-3.south east)
        -| (mat-13-3.north west);
        \draw[] (mat-18-3.north west) -| (mat-18-3.south east)
        -| (mat-18-3.north west);
      }]
      & R_{(p+r)+(q)}^{(n-1)+(m)} \arrow[r, "(-1)^{n-1}"]
      \arrow[ddd,"i_1"']
      & R_{(p+r)+(q-r-1)}^{(n-1)+(m+1)} \arrow[ddd, "i_1"]\\ \\ \\
      &R_{(p)+(q)}^{(n)+(m)} \arrow[d,phantom, "\oplus"{name=A}]
      & R_{(p)+(q-r-1)}^{(n)+(m+1)} \arrow[d, phantom, "\oplus"{name=B}]\\
      &R_{(p-1)+(q)}^{(n)+(m)} \arrow[ddd,"\pi_2"']
      & R_{(p-1)+(q-r-1)}^{(n)+(m+1)} \arrow[ddd, "\pi_2"] \\ \\ \\
      &R_{(p-r-1)+(q)}^{(n+1)+(m)} \arrow[r, "(-1)^{n+1}"]
      & R_{(p-r-1)+(q-r-1)}^{(n+1)+(m+1)} \\ \\ \\
      \arrow[from=A,to=B, "(-1)^n\id"]
      &R_{(p)+(q)}^{(n)+(m)} \arrow[r, "(-1)^n"] \arrow[dd]
      &R_{(p)+(q-r-1)}^{(n)+(m+1)} \arrow[dd]
      \\ \\
      &R_{(p-r-1)+(q)}^{(n+1)+(m)} \arrow[r, "(-1)^{n+1}"]
      &R_{(p-r-1)+(q-r-1)}^{(n+1)+(m+1)}
      \\ \\ \\
      R_{(p)+(q+r)}^{(n)+(m-1)} \arrow[r, "(-1)^ni_1"] \arrow[dd]
      &R_{(p)+(q)}^{(n)+(m)} \oplus\ R_{(p)+(q-1)}^{(n)+(m)}
      \arrow[r, "(-1)^n\pi_2"] \arrow[dd, "\id"]
      &R_{(p)+(q-r-1)}^{(n)+(m+1)}  \arrow[dd]\\ \\
      R_{(p-r-1)+(q+r)}^{(n+1)+(m-1)} \arrow[r, "(-1)^{n+1}i_1"]
      &R_{(p-r-1)+(q)}^{(n+1)+(m)} \oplus\ R_{(p-r-1)+(q-1)}^{(n+1)+(m)}
      \arrow[r, "(-1)^{n+1}\pi_2"]
      &R_{(p-r-1)+(q-r-1)}^{(n+1)+(m+1)}
    \end{tikzcd}
  \end{equation*}\caption{Components of the pushout: $\B_{r+1}(p,n)\otimes\Z_{r+1}(q,m)$, $\Z_{r+1}(p,n)\otimes\Z_{r+1}(q,m)$ and $\Z_{r+1}(p,n)\otimes\B_{r+1}(q,m)$. The central diagram maps via $k_1$ into the top diagram and via $k_2$ into the bottom.}
  \label{domainPushoutProductComponents}
\end{figure}
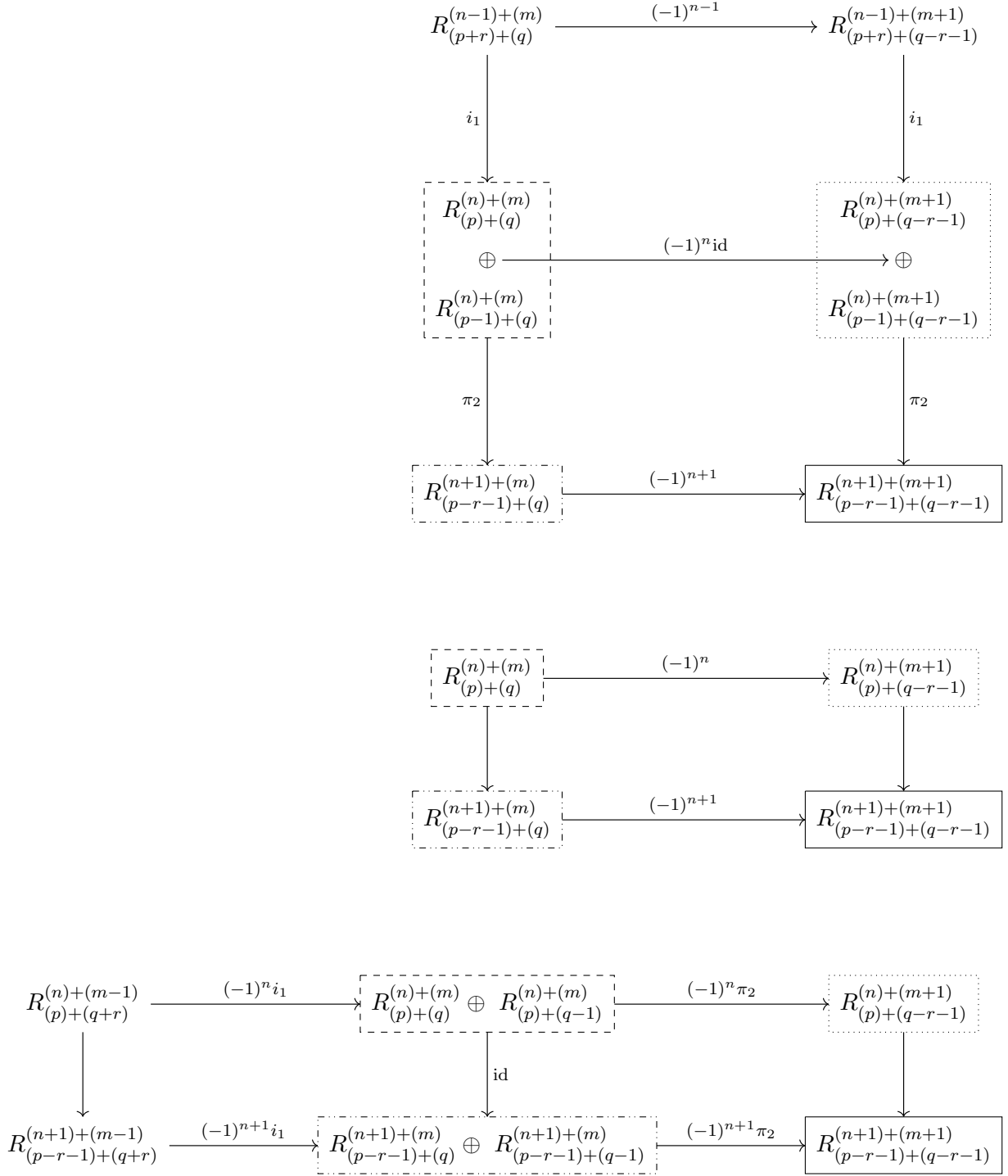

%% file: tex/diagrams/namedComponentsOfPushout.tex
\begin{figure}[t]
  \begin{equation*}
    \begin{tikzcd}[ampersand replacement=\&, column sep=2cm]
      \& R\left\{A\right\} \arrow[r, "(-1)^{n-1}"] \arrow[d] \& R\left\{D\right\} \arrow[d]\\
      \&\begin{matrix}R\left\{B\right\}\\\oplus\\R\left\{C\right\}\end{matrix} \arrow[d]\arrow[r,"(-1)^n\id"] \&
      \begin{matrix}R\left\{F\right\}\\\oplus\\R\left\{G\right\}\end{matrix}\arrow[d]\\
      \& R\left\{E\right\}\arrow[r,"(-1)^{n+1}"]\&R\left\{H\right\}\\
      \& R\left\{I\right\} \arrow[r,"(-1)^n"] \arrow[d] \& R\left\{K\right\} \arrow[d]\\
      \& R\left\{J\right\} \arrow[r,"(-1)^{n+1}"] \& R\left\{L\right\}\\
      R\left\{M\right\} \arrow[r,"(-1)^ni_1"]
      \arrow[d] \& R\left\{N\right\} \oplus R\left\{P\right\} \arrow[d] \arrow[r,"(-1)^n\pi_2"]
      \& R\left\{S\right\} \arrow[d]\\
      R\left\{O\right\} \arrow[r,"(-1)^{n+1}i_1"']
      \& R\left\{Q\right\} \oplus R\left\{T\right\} \arrow[r,"(-1)^{n+1}\pi_2"'] \& R\left\{U\right\}
    \end{tikzcd}
  \end{equation*}\caption{Named components of the pushout corresponding to \cref{domainPushoutProductComponents}}\label{namedDomainPushoutProductComponents}
\end{figure}

%% file: tex/diagrams/namedCodomainOfPushoutProduct.tex
\begin{figure}[h]
  \begin{equation*}
    \begin{tikzcd}[column sep=2cm]
      R\left\{\ul{A}\right\}\arrow[r,"(-1)^{n-1}i_1"]\arrow[d,"i_1"']&
      \begin{matrix}R\left\{\ul{C}\right\}\amsamp\oplus\amsamp R\left\{\ul{E}\right\}\end{matrix}
      \arrow[d]\arrow[r,"(-1)^{n-1}\pi_2"] & R\left\{\ul{J}\right\}\arrow[d,"i_1"]\\
      \begin{matrix}R\left\{\ul{B}\right\}\\\oplus\\R\left\{\ul{D}\right\}\end{matrix}\arrow[d,"\pi_2"']\arrow[r]&
      \begin{matrix}R\left\{\ul{F}\right\}\amsamp \oplus \amsamp R\left\{\ul{I}\right\}\\\oplus\amsamp\amsamp\oplus\\
        R\left\{\ul{H}\right\}\amsamp\oplus\amsamp R\left\{\ul{K}\right\}\end{matrix}\arrow[r]\arrow[d]
      &\begin{matrix}R\left\{\ul{M}\right\}\\\oplus\\R\left\{\ul{O}\right\}\end{matrix}\arrow[d,"\pi_2"]\\
      R\left\{\ul{G}\right\}\arrow[r,"(-1)^{n+1}i_1"']&
      \begin{matrix}R\left\{\ul{L}\right\}\amsamp\oplus\amsamp R\left\{\ul{N}\right\}\end{matrix}
      \arrow[r,"(-1)^{n+1}\pi_2"']&R\left\{\ul{P}\right\}
    \end{tikzcd}
  \end{equation*}\caption{Named codomain of the pushout-product $i\boxtimes j$}\label{namedCodomainPushoutProduct}
\end{figure}

%% file: tex/diagrams/ZChainsRepForZZ.tex
\begin{figure}
  \centering
  \begin{tikzpicture}[auto, xscale=3, yscale=0.8]
    \foreach \name/\x in {n+1+m+1/2, n+1+m/1, n+m/0} {
      \node at (\x,-1.2) [rotate=45] {\name};
    }
    \foreach \name/\y in {p-r-1+q-r-2/0, p-r-1+q-r-1/1,
      $\vdots$/2,p+q-r-1/3, $\vdots$/4,
      p+q/5,$\vdots$/6} {
      \node at (-1,\y) {\name};
    }
    \draw (-0.5,7.5) -- (-0.5,-0.3) -- (2.5,-0.3);
    \foreach \name/\y in {0/0, $\vdots$/1, $\vdots$/4,
      {$I$}/5, $\vdots$/6} {
      \node (0c\y) at (0,\y) {\name};
    }
    \foreach \name/\y in {0/0, $\vdots$/1, $\vdots$/2,
      {$J\oplus K$}/3, $\vdots$/4,
      $J \oplus K$/5, $\vdots$/6} {
      \node (1c\y) at (1,\y) {\name};
    }
    \foreach \name/\y in {0/0, {$L$}/1, $\vdots$/2, $L$/3,
      $\vdots$/6} {
      \node (2c\y) at (2,\y) {\name};
    }
    \draw[->] (0c5) to node
    {$\alpha_1$} (1c5);
    \draw[->] (1c3) to node
    {$\alpha_2$} (2c3);
  \end{tikzpicture}
  \caption{Representation in $\Zbb_\infty$-chains of $\Z_{r+1}(p,n)\otimes\Z_{r+1}(q,m)$ with differentials described in \cref{greek_i}}
  \label{ZCrossZInZ+Chains}
\end{figure}

%% file: tex/diagrams/ZChainsRepForZB.tex
\begin{figure}
  \centering
  \begin{tikzpicture}[auto, xscale=3, yscale=1]
    \foreach \name/\x in {n+1+m+1/2, n+1+m/1, n+m/0, n-1+m/-1} {
      \node at (\x,-1) [rotate=45] {\name};
    }
    \foreach \name/\y in {p-r-1+q-r-2/0, p-r-1+q-r-1/1,
      $\vdots$/2, p-r-1+q-1/3,p+q-r-1/4, $\vdots$/5,
      p+q-1/6,p+q/7,$\vdots$/8,p+q+r/9, $\vdots$/10} {
      \node at (-2,\y) {\name};
    }
    \draw (-1.5,11) -- (-1.5,-0.3) -- (2.5,-0.3);
    \foreach \name/\y in {0/0, $\vdots$/1, $\vdots$/8,
      $M$/9, $\vdots$/10} {
      \node (-1c\y) at (-1,\y) {\name};
    }
    \foreach \name/\y in {0/0, $\vdots$/1, $\vdots$/5,
      $\phantom{N\oplus\ }O\oplus P$/6,
      $N\oplus O\oplus P$/7, $\vdots$/8,
      $N\oplus O\oplus P$/9, $\vdots$/10} {
      \node (0c\y) at (0,\y) {\name};
    }
    \foreach \name/\y in {0/0, $\vdots$/1, $\vdots$/2,
      $\phantom{Q\oplus S\oplus\ }T$/3,
      $Q\oplus S\oplus T$/4,
      $\vdots$/5, $Q\oplus S\oplus T$/6, $Q\oplus S\oplus T$/7, $\vdots$/8,
      $\vdots$/10} {
      \node (1c\y) at (1,\y) {\name};
    }
    \foreach \name/\y in {0/0, $U$/1, $\vdots$/2, $U$/3,
      $U$/4, $\vdots$/5, $\vdots$/10} {
      \node (2c\y) at (2,\y) {\name};
    }
    \draw[->] (-1c9) to node
    {$\beta_1$} (0c9);
    \draw[->] (0c7) to node
    {$\beta_2$} (1c7);
    \draw[->] (0c6) to node [swap]
    {$\beta_3$} (1c6);
    \draw[->] (1c4) to node
    {$\beta_4$} (2c4);
    \draw[->] (1c3) to node [swap]
    {$\beta_5$} (2c3);
  \end{tikzpicture}
  \caption{Representation in $\Zbb_\infty$-chains of $\Z_{r+1}(p,n)\otimes\B_{r+1}(q,m)$ with differentials described in \cref{greek_i}}
  \label{ZCrossBInZ+Chains}
\end{figure}

%% file: tex/diagrams/ZChainsRepForBZ.tex
\begin{figure}
  \centering
  \begin{tikzpicture}[auto, xscale=3, yscale=1]
    \foreach \name/\x in {n+1+m+1/2, n+1+m/1, n+m/0, n-1+m/-1} {
      \node at (\x,-1) [rotate=45] {\name};
    }
    \foreach \name/\y in {p-r-1+q-r-2/0, p-r-1+q-r-1/1,
      $\vdots$/2, p-1+q-r-1/3,p+q-r-1/4, $\vdots$/5,
      p+q-1/6,p+q/7,$\vdots$/8,p+r+q/9, $\vdots$/10} {
      \node at (-2,\y) {\name};
    }
    \draw (-1.5,11) -- (-1.5,-0.3) -- (2.5,-0.3);
    \foreach \name/\y in {0/0, $\vdots$/1, $\vdots$/8,
      $A$/9, $\vdots$/10} {
      \node (-1c\y) at (-1,\y) {\name};
    }
    \foreach \name/\y in {0/0, $\vdots$/1, $\vdots$/5,
      $\phantom{B\oplus\ }C\oplus D$/6,
      $B\oplus C\oplus D$/7, $\vdots$/8,
      $B\oplus C\oplus D$/9, $\vdots$/10} {
      \node (0c\y) at (0,\y) {\name};
    }
    \foreach \name/\y in {0/0, $\vdots$/1, $\vdots$/2,
      $\phantom{E\oplus F\oplus\ }G$/3,
      $E\oplus F\oplus G$/4,
      $\vdots$/5, $E\oplus F\oplus G$/6, $E\oplus F\oplus G$/7, $\vdots$/8,
      $\vdots$/10} {
      \node (1c\y) at (1,\y) {\name};
    }
    \foreach \name/\y in {0/0, $H$/1, $\vdots$/2, $H$/3,
      $H$/4, $\vdots$/5, $\vdots$/10} {
      \node (2c\y) at (2,\y) {\name};
    }
    \draw[->] (-1c9) to node
    {$\gamma_1$} (0c9);
    \draw[->] (0c7) to node
    {$\gamma_2$} (1c7);
    \draw[->] (0c6) to node [swap]
    {$\gamma_3$} (1c6);
    \draw[->] (1c4) to node
    {$\gamma_4$} (2c4);
    \draw[->] (1c3) to node [swap]
    {$\gamma_5$} (2c3);
  \end{tikzpicture}
  \caption{Representation in $\Zbb_\infty$-chains of $\B_{r+1}(p,n)\otimes\Z_{r+1}(q,m)$ with differentials described in \cref{greek_i}}
  \label{BCrossZInZ+Chains}
\end{figure}

%% file: tex/diagrams/ZChainsRepForPushout.tex
\begin{landscape}
  \begin{figure}[hp]
    \centering
    \begin{tikzpicture}[auto, xscale=5, yscale=1.2]
      \foreach \name/\x in {n+1+m+1/2, n+1+m/1, n+m/0, n-1+m/-1} {
        \node at (\x,-1) [rotate=45] {\name};
      }
      \foreach \name/\y in {p-r-1+q-r-2/0, p-r-1+q-r-1/1,
        $\vdots$/2, p-1+q-r-1/3,p+q-r-1/4, $\vdots$/5,
        p+q-1/6,p+q/7,$\vdots$/8,p+r+q/9, $\vdots$/10} {
        \node at (-2,\y) {\name};
      }
      \draw (-1.5,11) -- (-1.5,-0.3) -- (2.5,-0.3);
      \foreach \name/\y in {0/0, $\vdots$/1, $\vdots$/8,
        $\ol{A}\oplus \ol{M}$/9, $\vdots$/10} {
        \node (-1c\y) at (-1,\y) {\name};
      }
      \foreach \name/\y in {0/0, $\vdots$/1, $\vdots$/5,
        $\phantom{\ol{B}\oplus\ }\ol{C}\oplus \ol{D}\oplus \ol{O}\oplus \ol{P}$/6,
        $\ol{B}\oplus \ol{C}\oplus \ol{D}\oplus \ol{O}\oplus \ol{P}$/7, /8,
        $\ol{B}\oplus \ol{C}\oplus \ol{D}\oplus \ol{O}\oplus \ol{P}$/9, $\vdots$/10} {
        \node (0c\y) at (0,\y) {\name};
      }
      \foreach \name/\y in {0/0, $\vdots$/1, $\vdots$/2,
        $\phantom{\ol{F}\oplus \ol{Q}} \ol{G}\oplus \ol{T}$/3,
        $\ol{F}\oplus \ol{Q}\oplus \ol{G}\oplus \ol{T}$/4,
        $\vdots$/5, $\ol{F}\oplus \ol{Q}\oplus \ol{G}\oplus \ol{T}$/6,
        $\ol{F}\oplus \ol{Q}\oplus \ol{G}\oplus \ol{T}$/7, /8,
        $\vdots$/10} {
        \node (1c\y) at (1,\y) {\name};
      }
      \foreach \name/\y in {0/0, $\ol{L}$/1, $\vdots$/2, $\ol{L}$/3,
        $\ol{L}$/4, $\vdots$/5, $\vdots$/10} {
        \node (2c\y) at (2,\y) {\name};
      }
      \draw[->] (-1c9) to node
      {\scriptsize${\scriptstyle \begin{pmatrix}1&(-1)^n\\0&(-1)^n\\(-1)^{n-1}&0\\0&1\\0&(-1)^{n+1}\end{pmatrix}}$} (0c9);
      \draw[->] (0c7) to node [label={[label distance=0.02cm]90:
      {\scriptsize${\scriptstyle \begin{pmatrix}(-1)^n&0&1&0&(-1)^n\\0&1&0&(-1)^{n+1}&0\\0&(-1)^n&0&0&(-1)^n\\0&1&0&0&1\end{pmatrix}}$}}] {} (1c7);
      \draw[->] (0c6) to node [swap,label={[label distance=0.02cm]270:
      {\scriptsize${\scriptstyle
 \begin{pmatrix}0&1&0&(-1)^n\\1&0&(-1)^{n+1}&0\\(-1)^n&0&0&(-1)^n\\1&0&0&1\end{pmatrix}}$}}] {} (1c6);
      \draw[->] (1c4) to node
      {\scriptsize${\scriptstyle \begin{pmatrix}0&0&1&(-1)^{n+1}\end{pmatrix}}$} (2c4);
      \draw[->] (1c3) to node [swap]
      {\scriptsize${\scriptstyle \begin{pmatrix}1&(-1)^{n+1}\end{pmatrix}}$} (2c3);
    \end{tikzpicture}
    \caption{Representation in $\Zbb_\infty$-chains of the pushout}
    \label{PushoutInZ+Chains}
  \end{figure}
\end{landscape}%


%% file: tex/diagrams/codomainOfPushoutProduct.tex
\begin{figure}
  \begin{equation*}
    \begin{tikzcd}
      R_{(p+r)+(q+r)}^{(n-1)+(m-1)} \arrow[rrr,"(-1)^{n-1}i_1"] \arrow[ddd,"i_1"] &&& R_{(p+r)+(q)}^{(n-1)+(m)} \arrow[r, phantom, "\oplus"{name=A}] & R_{(p+r)+(q-1)}^{(n-1)+(m)} \arrow[rrr, "(-1)^{n-1}\pi_2"] &&&
      R_{(p+r)+(q-r-1)}^{(n-1)+(m+1)} \arrow[ddd,"i_1"] \\ \\ \\
      R_{(p)+(q+r)}^{(n)+(m-1)} \arrow[d, phantom, "\oplus"{name=C}] &&& R_{(p)+(q)}^{(n)+(m)} \arrow[r, phantom, "\oplus"{name=B}] \arrow[d, phantom, "\oplus"{name=D}] & R_{(p)+(q-1)}^{(n)+(m)} \arrow[d, phantom, "\oplus"{name=E}]&&&
      R_{(p)+(q-r-1)}^{(n)+(m+1)} \arrow[d, phantom, "\oplus"{name=F}]\\
      R_{(p-1)+(q+r)}^{(n)+(m-1)} \arrow[ddd,"\pi_2"] &&& R_{(p-1)+(q)}^{(n)+(m)} \arrow[r, phantom, "\oplus"{name=G}] &
      R_{(p-1)+(q-1)}^{(n)+(m)} &&& R_{(p-1)+(q-r-1)}^{(n)+(m+1)} \arrow[ddd, "\pi_2"] \\ \\ \\
      R_{(p-r-1)+(q+r)}^{(n+1)+(m-1)} \arrow[rrr, "(-1)^{n+1}i_1"] &&& R_{(p-r-1)+(q)}^{(n+1)+(m)} \arrow[r, phantom, "\oplus"{name=H}] & R_{(p-r-1)+(q-1)}^{(n+1)+(m)} \arrow[rrr, "(-1)^{n+1}\pi_2"] &&& R_{(p-r-1)+(q-r-1)}^{(n+1)+(m+1)}
      \arrow[from=A, to=B, "id"]
      \arrow[from=C, to=D, "(-1)^nid"]
      \arrow[from=E, to=F, "(-1)^nid"]
      \arrow[from=G, to=H, "id"]
    \end{tikzcd}
  \end{equation*}\caption{Codomain of the pushout-product $i\boxtimes j$}\label{codomainPushoutProduct}
\end{figure}

%% file: tex/diagrams/pushoutProduct.tex
\begin{figure}
  \begin{equation*}
    \begin{tikzcd}[column sep=2cm]
      &R_{(p+r)+(q)}^{n+m-1} \arrow[r,"(-1)^{n-1}"]\arrow[d,"\alpha"]&R_{(p+r)+(q-r-1)}^{n+m}\arrow[d,"i_1"] \\
      R_{(p)+(q+r)}^{n+m-1}\arrow[d,"1"]\arrow[r,"\beta"]&
      \begin{matrix}R_{(p)+(q)}^{n+m}\amsamp\oplus\amsamp R_{(p)+(q-1)}^{n+m}\\
        \oplus\amsamp\amsamp\\R_{(p-1)+(q)}^{n+m}\end{matrix}\arrow[r,"\epsilon"]\arrow[d,"\gamma"]&
      \begin{matrix}R_{(p)+(q-r-1)}^{n+m+1}\\\oplus\\ R_{(p-1)+(q-r-1)}^{n+m+1}\end{matrix}\arrow[d,"\pi_2"]\\
      R_{(p-r-1)+(q+r)}^{n+m}\arrow[r,"(-1)^{n+1}i_1"]&
      \begin{matrix}R_{(p-r-1)+(q)}^{n+m+1}\amsamp\oplus\amsamp R_{(p-r-1)+(q-1)}^{n+m+1}\end{matrix}\arrow[r,"(-1)^{n+1}\pi_2"]
      \arrow[ddd, "i\boxtimes j", dashed]
      &R_{(p-r-1)+(q-r-1)}^{n+m+2}\\\\\\
      R_{(p+r)+(q+r)}^{(n-1)+(m-1)}\arrow[r,"(-1)^{n-1}i_1"]\arrow[d,"i_1"]
      &\begin{matrix}R_{(p+r)+(q)}^{(n-1)+(m)}
        \amsamp\oplus\amsamp R_{(p+r)+(q-1)}^{(n-1)+(m)}\end{matrix}
      \arrow[r,"(-1)^{n-1}\pi_2"]\arrow[d]
      &R_{(p+r)+(q-r-1)}^{(n-1)+(m+1)}\arrow[d,"i_1"]\\
      \begin{matrix}R_{(p)+(q+r)}^{(n)+(m-1)}\\\oplus\\R_{(p-1)+(q+r)}^{(n)+(m-1)}\end{matrix}
      \arrow[r,"(-1)^n"]\arrow[d,"\pi_2"]&
      \begin{matrix}R_{(p)+(q)}^{(n)+(m)}\amsamp\oplus\amsamp R_{(p)+(q-1)}^{(n)+(m)}
        \\\oplus\amsamp\amsamp\oplus\\
        R_{(p-1)+(q)}^{(n)+(m)}\amsamp\oplus\amsamp R_{(p-1)+(q-1)}^{(n)+(m)}\end{matrix}
      \arrow[r,"(-1)^n"]\arrow[d]
      &\begin{matrix}R_{(p)+(q-r-1)}^{(n)+(m+1)}\\\oplus\\R_{(p-1)+(q-r-1)}^{(n)+(m+1)}\end{matrix}
      \arrow[d,"\pi_2"]\\
      R_{(p-r-1)+(q+r)}^{(n+1)+(m-1)}\arrow[r,"(-1)^{n+1}i_1"]&
      \begin{matrix}R_{(p-r-1)+(q)}^{(n+1)+(m)}\amsamp\oplus\amsamp R_{(p-r-1)+(q-1)}^{(n+1)+(m)}\end{matrix}
      \arrow[r,"(-1)^{n+1}\pi_2"]&
      R_{(p-r-1)+(q-r-1)}^{(n+1)+(m+1)}
    \end{tikzcd}
  \end{equation*}\caption{The pushout-product of generating cofibrations}\label{pushoutProduct}
\end{figure}

%% file: tex/diagrams/namedPushoutProduct.tex
\begin{figure}[h]
  \begin{equation*}
    \begin{tikzcd}[column sep=3cm]
      & R\left\{\ol{A}\right\} \arrow[r,"(-1)^{n-1}"]\arrow[d,"\alpha"]& R\left\{\ol{D}\right\}\arrow[d,"i_1"]\\
      R\left\{\ol{M}\right\}\arrow[r,"\beta"]\arrow[d,"1"']
      &\begin{matrix}R\left\{\ol{B}\right\}\amsamp\oplus\amsamp R\left\{\ol{P}\right\}\\
         \oplus\amsamp\amsamp\\
         R\left\{\ol{C}\right\} \amsamp\amsamp
       \end{matrix}\arrow[r,"\epsilon"]\arrow[d,"\gamma"]
       &\begin{matrix}R\left\{\ol{F}\right\}\\\oplus\\R\left\{\ol{G}\right\}\end{matrix}\arrow[d,"\pi_2"]\\
       R\left\{\ol{O}\right\}\arrow[r,"(-1)^{n+1}i_1"']&
       \begin{matrix}R\left\{\ol{Q}\right\}\amsamp\oplus\amsamp R\left\{\ol{T}\right\}\end{matrix}\arrow[r,"(-1)^{n+1}\pi_2"']
       \arrow[ddd, dashed, "i\boxtimes j"]& R\left\{\ol{L}\right\}\\\\\\
       R\left\{\ul{A}\right\}\arrow[r,"(-1)^{n-1}i_1"]\arrow[d,"i_1"']&
       \begin{matrix}R\left\{\ul{C}\right\}\amsamp\oplus\amsamp R\left\{\ul{E}\right\}\end{matrix}
       \arrow[d]\arrow[r,"(-1)^{n-1}\pi_2"] & R\left\{\ul{J}\right\}\arrow[d,"i_1"]\\
       \begin{matrix}R\left\{\ul{B}\right\}\\\oplus\\R\left\{\ul{D}\right\}\end{matrix}\arrow[d,"\pi_2"']\arrow[r,"(-1)^n"]&
       \begin{matrix}R\left\{\ul{F}\right\}\amsamp \oplus \amsamp R\left\{\ul{I}\right\}\\\oplus\amsamp\amsamp\oplus\\
         R\left\{\ul{H}\right\}\amsamp\oplus\amsamp R\left\{\ul{K}\right\}\end{matrix}\arrow[r,"(-1)^n"]\arrow[d]
       &\begin{matrix}R\left\{\ul{M}\right\}\\\oplus\\R\left\{\ul{O}\right\}\end{matrix}\arrow[d,"\pi_2"]\\
       R\left\{\ul{G}\right\}\arrow[r,"(-1)^{n+1}i_1"']&       \begin{matrix}R\left\{\ul{L}\right\}\amsamp\oplus\amsamp R\left\{\ul{N}\right\}\end{matrix}
       \arrow[r,"(-1)^{n+1}\pi_2"']&R\left\{\ul{P}\right\}
     \end{tikzcd}
   \end{equation*}\caption{Named generators of the pushout-product with differentials described in \cref{namedDifferentialsOfPushoutroduct} and the map $i\boxtimes j$ induced by \cref{mapOnGeneratorInducedBy}}\label{ABCPushoutDiagram}
 \end{figure}